\documentclass[english, 12pt]{amsart}
\usepackage{amsmath,amssymb,amsthm}
\usepackage{mathtools}
\usepackage{hyperref}
\usepackage{cleveref}
\usepackage{booktabs}
\usepackage{tikz-cd}
\usepackage{todonotes}
\usepackage{derivative}

\newcommand{\NN}{\mathbb{N}}
\newcommand{\ZZ}{\mathbb{Z}}
\newcommand{\QQ}{\mathbb{Q}}
\newcommand{\RR}{\mathbb{R}}
\newcommand{\CC}{\mathbb{C}}
\newcommand{\FF}{\mathbb{F}}

\newcommand{\cO}{\mathcal{O}}

\newcommand{\cM}{\mathcal{M}}

\newcommand{\cI}{\mathcal{I}}

\newcommand{\bA}{\mathbb{A}}

\newcommand{\lang}{\mathcal{L}}
\newcommand{\Lval}{\lang_{\text{val}}}

\newcommand{\Lring}{\lang_{\text{ring}}}

\newcommand{\LRV}{ \lang^{\RV} }

\newcommand{\Log}{\lang_{\text{og}}}

\newcommand{\Ldp}{\lang_{\text{DP}}}
\newcommand{\THen}{T_{\text{Hen}}}

\newcommand{\LgDP}{ \lang_{\text{gDP}} }

\newcommand{\cC}{\mathcal{C}}

\newcommand{\cK}{\mathcal{K}}

\newcommand{\THeno}{T_{\text{Hen},0}}

\newcommand{\RVprod}{ \RV_{\bar{n}} }

\DeclareMathOperator{\Loc}{Loc}

\DeclareMathOperator{\graph}{graph}
\DeclareMathOperator{\ord}{ord}
\DeclareMathOperator{\dom}{dom}
\DeclareMathOperator{\ac}{ac}
\DeclareMathOperator{\rv}{rv}

\DeclareMathOperator{\Th}{Th}

\DeclareMathOperator{\RF}{RF}
\DeclareMathOperator{\VG}{VG}
\DeclareMathOperator{\VF}{VF}

\DeclareMathOperator{\ind}{ind}

\DeclareMathOperator{\lct}{lct}

\DeclareMathOperator{\Gal}{Gal}
\DeclareMathOperator{\RV}{RV}

\DeclarePairedDelimiter{\abs}{\lvert}{\rvert}

\newtheorem{theorem}{Theorem}[section]
\newtheorem{lemma}[theorem]{Lemma}
\newtheorem{proposition}[theorem]{Proposition}
\newtheorem{corollary}[theorem]{Corollary}

\theoremstyle{definition}
\newtheorem{definition}[theorem]{Definition}

\theoremstyle{remark}
\newtheorem{remark}[theorem]{Remark}
\newtheorem{example}[theorem]{Example}

\newtheorem{notation}[theorem]{Notation}

\usepackage{color}

\title[Existential $p$-adic integration and descent]{Existential uniform $p$-adic integration and descent for integrability and largest poles}

\author%[Cluckers]
{Raf Cluckers}
\address{Univ.~Lille,
CNRS, UMR 8524 - Laboratoire Paul Painlevé, F-59000 Lille, France, and
KU Leuven, Department of Mathematics, B-3001 Leu\-ven, Bel\-gium}
\email{Raf.Cluckers@univ-lille.fr}
\urladdr{http://rcluckers.perso.math.cnrs.fr/}

\author%[Stout]
{Mathias Stout}
\address{KU Leuven, Department of Mathematics, B-3001 Leu\-ven, Bel\-gium}
\email{mathias.stout@kuleuven.be}

\subjclass[2020]{Primary 03C98, 11U09, 14B05; Secondary 11S40,  14E18, 11F23}
\keywords{Descent for integrability of $p$-adic integrals, cell decomposition, Igusa's local zeta functions, Serre-Poincar\'e series, log canonical threshold, oscillation index, semi-continuity of $p$-adic integrability indices}

\thanks{The authors would like to thank A.~Aizenbud, Y.~Hendel, F.~Loeser, and E.~Sayag for valuable discussions around the themes of this paper. Many thanks also to J.~Denef, for discussions on controlling quantifiers in relation to motivic and $p$-adic integrals. The author R.~C.\  was partially supported by the Labex CEMPI  (ANR-11-LABX-0007-01).
}

\begin{document}
\begin{abstract}
Since the work by Denef, $p$-adic cell decomposition provides a well-established method to study $p$-adic and motivic integrals.
In this paper, we present a %precise
variant of this method that keeps track of existential quantifiers.
%For a certain class of definable functions, including non-negative powers of polynomials maps, we show the following: if $f$ is integrable with respect to the Haar measure on a $p$-adic field $L$, then its restriction to any $p$-adic subfield $K$ is also integrable (with respect to the corresponding Haar measure).
This enables us to deduce descent properties for $p$-adic integrals. In particular, we show that integrability for `existential' functions descends from any  $p$-adic field to any $p$-adic subfield.
As an application, we obtain that the largest pole of the Serre-Poincar\'e series can only increase when passing to field extensions.
As a side result, we prove a relative quantifier elimination statement for Henselian valued fields of characteristic zero that preserves existential formulas.

\end{abstract}

\maketitle

%\tableofcontents

\section{Introduction} \label{sec:intro}

%Let us start with some questions.

This paper centers around the following question. Given two $p$-adic fields  $L\geq K$, both equipped with the Haar measure, and two `similar' functions $f_L$ on $L^n$ and $f_K$ on $K^n$, in what generality can one deduce the integrability of $f_K$ from the integrability of $f_L$? This turns out to be a delicate matter, with the geometric hope that in $L$ one sees `more' than in $K$. In this paper we make precise what `similar' and `more' may mean for this question about descent %of integrability
from $L$ to $K$.

%Classically,
In the context of uniform $p$-adic %and motivic
integration, one often describes $f_K$ uniformly over all $p$-adic fields $K$ by using model theoretic tools, see e.g.~\cite{Pas89,CH18}. This  leads to uniform descriptions of integrands, their integrability, their integrals, loci of integrability, etc., when varying over all $K$, but the mentioned descent for integrability is more subtle than just having these uniform descriptions. We will use two levels of extra control %on the descriptions of the $f_K$
that can both be seen as forms of positivity. The first is a non-negativity notion coming simply from avoiding differences, which is reminiscent of the semi-ring approach to motivic functions from \cite{CL08}.
%, more precisely, we will use the formally non-negative functions  $f_K$ from \cite{CGH21}.
The second form of positivity comes from allowing only existential formulas in the descriptions of the $f_K$. Note that any existential formula $\varphi$ %in any language
has the following form of positivity: for any substructures $A\subseteq B$, one has an inclusion of the solution sets $\varphi(A)\subseteq \varphi(B)$, see %Property
(\ref{eq:exist}). %, hence, counting in finite structures can only grow when the structure becomes larger.
%Let us explain our result
The study of existential $p$-adic integrals and their descent properties are new.
Let us make this all precise.

\subsection{}
We first introduce our results on descent in the concrete setting of point counting in finite rings as in Igusa's work, see e.g.~\cite{Den91}.
Let $f(x) \in \ZZ[x]$ be a non-constant polynomial in $m$ variables.
If we wish to understand the number of solutions of $f(x) = 0 \mod N$ for all $N>0$, it suffices by the Chinese remainder theorem to understand the numbers
\[  \tilde{N}_{n,p}(f) \coloneqq \# \{ \xi \in (\ZZ_p/(p^n))^m \mid f(\xi) = 0 \}, \]
for primes $p$ and  integers $n\ge 0$.
One way of studying these numbers is by investigating the corresponding \emph{Poincar\'e series}
\[ \tilde{P}_{f,p}(T) = \sum_{n = 0}^{+ \infty} \frac{ \tilde{N}_{n,p}(f) } { p^{nm} } T^n.  \]
Up to a transformation $T \mapsto p^{-s}$, we may view $\tilde{P}_{f,p}(p^{-s})$ as a function of one complex variable $s$ with positive real part.
Igusa showed in \cite{Igu75} that this function can be expressed in terms of a $p$-adic integral and that $\tilde{P}_{f,p}(p^{-s})$ is a rational function in $p^{-s}$, with a candidate description of all the poles based on a log resolution of $f = 0$. The maximum of the real parts of the complex poles of $\tilde{P}_{f,p}(p^{-s})$ is of particular interest, as it relates to the asymptotic growth of the quantities $\tilde{N}_{n,p}(f)$, for $n \to \infty$.

%There is , and, the maximum of the real parts of the poles of $\tilde{P}_{f,p}(p^{-s})$

Now consider a finite extension $K$ of $\QQ_p$ with valuation ring $\cO_K$ and choose a uniformizer $\pi \in K$. Then we can similarly count the number of solutions $\tilde{N}_{n,K}(f)$ of $f(x) = 0$ in $\cO_{K}^m/(\pi^n)$, where $(\pi^n)$ stands for the ideal generated by $\pi^n$ in $\cO_{K}^m$ . Denote by $q$ the cardinality of the residue field of $K$ and consider the corresponding Poincar\'e series
\[ \tilde{P}_{f,K}(T) = \sum_{n = 0}^{+ \infty} \frac{ \tilde{N}_{n,K}(f) }{ q^{nm} } T^n. \]
Again, $\tilde{P}_{f,K}(q^{-s})$ is rational in $q^{-s}$, and, from the  candidate description of poles based on log resolutions, it follows that for the largest real part of the poles of $\tilde{P}_{f,K}(q^{-s})$, denoted by $- \tilde{\lambda}_K(f)$, one has
\[  \lct(f) \leq \tilde{\lambda}_K(f), \]
where the left-hand side is the log canonical threshold of $f(x)$; here, by convention, if $\tilde{P}_{f,K}( q^{-s} )$ has no poles, then $\tilde{\lambda}_K(f) = + \infty$.

One corollary of our results is that when $L$ is any finite field extension of $K$, then
\begin{equation}\label{eq:lambda}
 \tilde{\lambda}_{L}(f) \leq \tilde{\lambda}_{K}(f)  .
 \end{equation}
We refer to this phenomenon as descent: information for $\tilde{\lambda}_K(f)$ descends from field extensions.
It is conceivable that the above inequality (\ref{eq:lambda}) can also be proven by purely algebraic arguments, but the main results of this paper deduce such descent properties in contexts where algebro-geometric tools are not readily available, as is already the case for Serre's variant which we explain next.

\subsection{}
Let us explain the extent of our results with Serre's variant of the above Poincaré series.
Continuing with our notation for $f,K,\pi,n$, consider the numbers
\[ N_{n,K}(f) \coloneqq \# \{ \xi \in \cO_K^m/(\pi^n)\mid \exists y \in \cO_{K}^m ( f(y) = 0 \land \xi = y + (\pi^n) ) \} \]
and the corresponding \emph{ Serre-Poincar\'e} series
\[ P_{f,K}(T) \coloneqq \sum_{n = 0}^{+\infty} \frac{N_{n,K}(f)}{q^{n(m+1)}}  T^n .\]
Denef \cite{Den84} proved that $P_{f,K}(q^{-s})$ is rational in $q^{-s}$ using model-theoretic ingredients. The first step of the proof consists again of expressing $P_{f,K}(q^{-s})$ as a certain $p$-adic integral. The main complication compared to the case of the usual Poincar\'e series is that the data from which this integral is built is no longer purely algebraic, rather, it is definable (in e.g.~the Denef-pas language), and more precisely, existentially definable (see Definition \ref{def:existent}).

By taking the model-theoretic viewpoint, we can treat the Poincar\'e and Serre-Poincar\'e series on the same footing, leading to similar descent results, as follows. % for the Poincar\'e series also apply to the .
%Let us sketch our results and approach.
Write $-\lambda_{K}(f)$ for the largest real part of the poles (in $s$) of $P_{f,K}(q^{-s})$ (and put $\lambda_K(f) = + \infty$ if there are no poles). We prove the following:
\begin{theorem}(Descent for the Serre-Poincar\'e series) \label{th:intro-Serre-Poincare}
If $K$ is a $p$-adic field, then for any finite field extension $L$, it holds that
\[ \lambda_{L}(f) \leq \lambda_{K}(f). \]
\end{theorem}
As in the case of the Poincar\'e series, this yields a comparison between the asymptotics of the $N_{n,L}(f)$ and the $N_{n,K}(f)$. As one always has $\lambda_{K}(f)\ge 0$, this theorem opens the way to study properties of limit values of $\lambda_{K}(f)$ when $K$ becomes a bigger and bigger extension. % of some $\QQ_p$.

\subsection{}
Our applications to the poles of the Poincar\'e and Serre-Poincar\'e series are two special cases of a general result which we call descent for the $K$-index. We now explain this in more detail.

Using model theory, one can interpret functions uniformly in field extensions of $K$. More precisely, we consider families of functions $f =(f_K)_K$ where $K$ runs over all $p$-adic fields, such that there exists a single formula (in the sense of logic) of which each $f_K$ is the interpretation in $K$. For such $f = (f_K \colon X_K\subseteq K^m \to K)_K$, we define, for each $p$-adic field $K$,
%\Raf{Needed?: where $X_K$ is bounded,}
the $K$-index of $\abs{f}$ as
\[ \ind^K_X(\abs{f}) \coloneqq \sup \left\{ s \in \RR_{> 0} \, \middle| \,  \int_{X_K} \abs{f_K}_K^s \, dx < + \infty \right\}, \]
with the convention that $\ind^K_X(\abs{f}) = 0$ when the set of such $s$ is empty and $\ind^K_X(\abs{f}) = + \infty$ when it is unbounded.

Here $\abs{\cdot}_K$ is the absolute value associated to the $\pi$-adic valuation (with $\abs{\pi}_K = q^{-1}$) and the integral is taken with respect to the (additive) Haar measure $\mu$, normalized such that $\mu(\cO_K) = 1$.
%Note that if $f_K(x) = \frac{1}{p(x)}$ almost everywhere for each $K$ and some fixed polynomial $p(x) \in \ZZ[x]$ with $p(0) = 0$ and each $X_K = \cO_{K}$, then $\ind^K_X(\abs{f})$ coincides with the $K$-log canonical threshold $\lambda_K((p))$ as considered in \cite[Thm.~2.7]{VZ07}.

Descent for the $K$-index of $\abs{f}$ may fail for trivial reasons, since one may easily define functions which become ``less singular'' in field extensions (Example \ref{ex:descent-failure}). However, it does hold when the graph of $f$ is existentially definable.
\begin{theorem}(Descent for the $K$-index) \label{th:intro-descent}
Let $f = (f_K \colon X_K \to K)_K$ be a definable function in the language $\Lval = \Lring \cup \{\cO\}$. If the graph of $f$ is given by an existential formula\footnote{An existential formula is a formula of the form $\exists y \varphi(x,y)$ for some quantifier-free formula $\varphi(x,y)$ and tuples of variables $x$, $y$, see Definition \ref{def:existent}}, then for any $p$-adic field $K$ and any finite field extension $L \geq K$, we have
\[ \ind^L_X(\abs{f}) \leq \ind^K_X(\abs{f}) .\]
\end{theorem}
Our results for the (Serre)-Poincar\'e series follow by identifying their largest poles with $-\ind^K_X(\abs{f})$, for some suitable $f$.

\subsection{}
Our main technical tool consists of a cell decomposition statement (Theorem \ref{th:CDIII}), with precise control on existential quantifiers.

The idea of using cell decomposition to study $p$-adic integrals has a long and successful history. It was first introduced by Denef in \cite{Den84} and then further developed by Pas in \cite{Pas89,Pas90}. Since a cell decomposition gives a nice description of a given set, uniformly in a certain class of fields, it is well suited to the development of theories of uniform $p$-adic integration over either local fields of sufficiently large residue field  characteristic (e.g. \cite{CGH14,CGH18}) or $p$-adic fields of any residue field characteristic (\cite{CH18}), both following the more abstract motivic approach of \cite{CL08,CL10}.

We are interested in comparing the (integrals of) interpretations $f_K$ and $f_L$ of a single $f = (f_K)_K$ between $p$-adic field $K$ and a finite field extension $L \geq K$. This leads to two key technical differences compared to the usual cell decomposition powering the aforementioned frameworks for uniform $p$-adic integration. First, we work in a language with leading term maps, rather than angular component maps. This is because leading term maps always extend to field extensions (Proposition \ref{le:L-sub}), but angular component maps might not (Proposition \ref{prop:ac-restriction-failure}). Second, we expand our language by certain predicates related to Hensel lifts in order to have more precise control on existential quantifiers.
%This leads to the existential uniform $p$-adic integration considered in Section \ref{sec:p-E-int}.

Roughly speaking, this control on quantifiers allows us to split up an existentially definable set in cells that are also existentially definable (see Theorem \ref{th:CDIII} for more details). This then allows us to reduce the proof Theorem \ref{th:intro-descent} to the case where $X$ is a cell on which $f$ is prepared in a certain way. The importance of existential formulas is related to the following fact: if $X$ and the graph of $f = (f_K \colon X_K \to K)_K$ are existentially definable, then whenever $K \leq L$, we have
\begin{equation}\label{eq:exist}
X_K \subseteq X_L \mbox{ and }f_K = {f_L}_{|X_K}.
\end{equation}
This allows for a meaningful comparison of the integrals of $f_K$ and $f_L$.
\subsection{}
Additionally, our cell decomposition implies an existential quantifier reduction statement in a certain language $\LRV$(introduced in Section \ref{sec:conventions}). It contains a sort $\VF$ for the valued field as well as sorts $\RV_N$, for the leading term structures $K^{\times}/(1 + N \cM_{K}) \cup \{0\}$, for all integers $N>0$. It also contains the aforementioned Hensel lift predicates. In this language, we may formulate the theory $\THeno$ of characteristic zero Henselian valued fields (with arbitrary residue field characteristic).
\begin{theorem}[$\exists$-VF elimination] \label{th:E-reductionv1}
Any existential $\LRV$-formula is equivalent modulo $\THeno$ to an existential $\LRV$-formula without any valued field-quantifiers.
\end{theorem}
Because we include extra Hensel lift predicates in our language, we are able to obtain a tighter control on quantifiers than in \cite[Prop.~4.3]{Flen11}. Additionally, our quantifier reduction statement implies an AKE-like principle for the existential theories of Henselian valued fields (Corollary \ref{cor:E-AKE}).
\section{Valued fields and leading terms} \label{sec:conventions}

\subsection{Notation and conventions} \label{sec:notation}
Throughout this text, $K$ denotes a nontrivial valued field, with valuation ring $\cO_K \neq K $.
This means that $x \in K$ one has $x \in \cO_{K}$ or $x^{-1} \in \cO_{K}$, and $K$ is the fraction field of $\cO_K$. Write $\cM_K$ for the unique maximal ideal of $\cO_{K}$ and $\Gamma = K^{\times}/\cO_{K}^{\times}$ for the value group of $K$. The value group $\Gamma$ will be considered as an additive group and may be of any rank.
It comes equipped with a surjective valuation map $\ord \colon K^{\times} \to \Gamma$.

For each integer $N>0$, write $R_N$ for the \emph{residue ring} $\cO_{K}/(N \cM_{K})$.
Also, define an \emph{open ball} $B(a,\gamma)$, with center $a \in K$ and valuative radius $\gamma \in \Gamma$ as
\[ B(a,\gamma) \coloneqq \{ x \in K \mid \ord(x - a) > \gamma \}. \]

We will make extensive use of the \emph{leading term} structures $\RV_N$.
For any $N >0$ define $\RV_N^{\times}$ as the multiplicative group
\[\RV_N^{\times} \coloneqq K^{\times}/(1 + N \cM_{K}),\]
and set $\RV_N = \RV_N^{\times} \cup \{ 0\}$.
We have maps $\rv_N \colon K \to \RV_N$, given by the natural projection on $K^{\times}$ extended by $\rv_N(0) = 0$.
Whenever $N $ divides $M$ for some $0<N\le M$, the map $\rv_N$ induces a map $\RV_M \to \RV_N$, which we also denote by $\rv_N$.
When $N =1$, simply write $\RV$ and $\rv$ instead of $\RV_1$ and $\rv_1$.
\begin{example}
Let $k$ be any field and let $K$ be $k((t))$, with valuation ring $\cO_K = k[[t]]$.
For $a_j \in k \setminus \{0\}$, one has that
\[ \rv\left( \sum_{i \geq j} a_i t^i \right) = a_j t^j (1 + \cM_K). \]
\end{example}
In the above example, we have that $\RV_N^{\times} \cong R_N^{\times} \times \Gamma$ for all integers $N >0$.
This is not always the case, since angular component maps (recalled in Section \ref{sec:acN}) do not always exist, \cite{Pas90b}. However, there always exists natural short exact sequences
\[  \{1\} \to R_N^{\times} \to \RV_N^{\times}  \xrightarrow{\ord} \Gamma \to \{1\}. \]
Thus, $\RV_N$ combines information about the residue ring $R_N$ and value group $\Gamma$.
Furthermore, the valuation map $\ord \colon K^{\times} \to \Gamma$ can be seen to factor through the map $\ord \colon \RV_N^{\times} \to \Gamma$, which we denote with the same name.
In particular, the following defintion makes sense.
\begin{definition} \label{def:divRV}
Define a binary relation relation $|$ on $\RV_N$, given by
\[ \rv_N(x) | \rv_N(y) \leftrightarrow \ord(x) \leq \ord(y). \]
\end{definition}

Finally, we note the following lemma whose proof is immediate (see also \cite[Prop.~2.2]{Flen11} for further equivalent conditions).
\begin{lemma} \label{le:RV-equality}
For $x,y \in K \setminus \{0\}$, we have
\[ \rv_N(x) = \rv_N(y) \leftrightarrow \ord(x-y) > \ord y + \ord N.\]
\end{lemma}

%\begin{remark}
%More generally, for any nonzero $\delta \in \cO_K$, we can define
%%
%\[\RV_{\delta} \coloneqq K^{\times}/(1 + \{ x \in K \mid \ord x > \ord \delta \}) \cup \{0\}.\]
%%
%%With this notation, we have $\RV_{\ord N} = \RV_{ N}$ (with $\RV_N$ as previously defined).
%As we only need the $\RV_N$, we will not consider these more general semigroupss.
%We do note however that all the arguments of Section \ref{sec:RV} go through essentially unchanged for arbitrary $\RV_{\delta}$.
%\end{remark}
%%
\subsection{Partial addition on $\RV_N$} \label{sec:RV}
We now recall some preliminaries on the leading term structures and their partial addition in particular. None of the mentioned facts are new. We include them here to keep the paper self-contained. We also refer to \cite{Flen11} for a further overview of some basic properties.

\begin{definition} \label{def:oplus}
The partial addition $\oplus$ on $\RV_N$ is a ternary relation such that $\oplus(\xi_1,\xi_2,\xi_3)$ holds if and only if there exists $x_i \in K$ with $\rv_N(x_i) = \xi_i$ for $i = 1,2,3$ such that $x_1 + x_2 = x_3$.
\end{definition}
Instead of viewing $\oplus$ as a relation, we can (and will) equivalently consider it as a binary operation $+$, which takes two elements $\xi_1,\xi_2 \in \RV_N$ and produces a set of elements
\[
\xi_1 + \xi_2 \coloneqq  \{ \xi_3 \in \RV_N\mid \oplus(\xi_1,\xi_2,\xi_3) \}.
\]
For two subsets $A,B \subseteq \RV_N$, we define their sum as
\[
A + B \coloneqq \bigcup_{ \substack{ \xi_1 \in A \\ \xi_2 \in B }  } \xi_1 + \xi_2.
\]
\begin{notation} \label{re:plus-singleton}
If for $a,b \in \RV_{N}$ and $c \in \RV_{N}$ it holds that $ a + b = \{c\}$, then we abbreviate this by $ a + b = c$.
Similarly, for $a,b \in \RV_{NM}$ and $c \in \RV_N$, we write $\rv_N(a + b) = c$ instead of $\rv_N(a+ b) = \{c\}$.
In this paper we write $\NN$ for the set of positive integers.
\end{notation}
Together with the (multiplicative) group structure of $\RV_N^{\times}$ and the map $\ord \colon \RV_N \to \Gamma$, this ``addition'' endows $\RV_N$ with the structure of a \emph{valued hyperfield} (\cite{Kra83} \cite[Prop.~1.17]{LT22}).
In particular, we have the following properties:
\begin{lemma}
	Let $N \in \NN$ and $\xi_1,\xi_2,\xi_3 \in \RV_N$. Then the partial addition on $\RV_N$ satisifies
	\begin{enumerate}
		\item (neutral element) $0 + \xi_1  = \xi_1 + 0 = \xi_1$,
		\item (commutativity) $\xi_1 + \xi_2 = \xi_2 + \xi_1$,
		\item (associativity) $( \xi_1 + \xi_2) + \xi_3 = \xi_1 + (\xi_2 + \xi_3)$.
	\end{enumerate}
\end{lemma}
\begin{proof}
	This follows from unwinding the definitions. For example, $\xi_4 \in \xi_1 + (\xi_2 + \xi_3)$ if and only there exists $x_i \in K$ for $i = 1,2,3,4$ with the property that $\rv_N(x_i) = \xi_i$ and $x_1 + x_2 + x_3 = x_4 $. This proves associativity.
\end{proof}
By the above lemma, expressions such as $\sum_{i = 0}^d \xi_i$ with $\xi_i \in \RV_N$ are well-defined.
\begin{lemma} \label{le:RV-add}
Consider nonzero integers $N,d$ and elements $\xi_i \in \RV_N$ for $i = 0,\dots,d$. Let $i_0$ be such that $\ord \xi_{i_0} = \min_i \ord \xi{_i}$. Then for any $a \in K$  with $\rv_N(a) \in \sum_{i = 0}^d \xi_i$ one has
$$
\sum_{i = 0}^d\xi_i = \rv_N B(a, \ord \xi_{i_0} + \ord N ).
$$
%	\begin{enumerate}
%		\item $0 \in \sum_{i=0}^d \xi_i$ if and only if $\sum_{i = 0}^d \xi_i = \rv_N B(0,\ord \xi_{i_0} + \ord N)  $ if and only if $\sum_{i=0}^d \xi_i$ contains an element $\zeta$ with $\ord \zeta > \ord \xi_{i_0} + \ord N$,
%		%\item $\sum_{i = 0}^d \xi_i$ is a singleton if and only if $ \ord( \sum_{i = 0}^d \xi_i ) = \{ \ord \xi_{i_0} \} $,
%		\item If $0 \notin \sum_{i =0}^d \xi_i$, then for any $a \in K$, such that $\rv_N(a) \in \sum_{i = 0}^d \xi_i$ we have $\sum_{i = 0}^d\xi_i = \rv_N B(a, \ord \xi_{i_0} + \ord N )$.
%	\end{enumerate}
\end{lemma}
%
%\begin{proof}
%	We only consider the second case; the first case is similar. Let $x_0,\dots,x_d \in K$ be such that $\rv_N(x_i) = \xi_i$ and let $a = \sum_{i = 0}^d x_i$.
%	If $b \in K $ is such that $\rv_N(b) \in \sum_{i = 0}^d \xi_i$, then there exists $y_0,\dots,y_d$ such that $\ord(y_i - x_i) > \ord x_i + \ord N $ for $i = 0,\dots,d$ and $b = \sum_{i = 0}^d y_i$. Then we have that
%	%
%	\[
%	\ord(b - a ) \geq \min_{i} \ord(y_i - x_i) > \ord \xi_{i_0} + \ord N.
%	\]
%	%
%	Conversely, if $\ord (b-a) > \ord \xi_{i_0} + \ord N$, define $y_{i_0} \coloneqq b - \sum_{i\neq i_{0}} x_i$.
%	We thus have $b = \sum_{i \neq i_0} x_i + y_{i_0}$ and
%	%
%	\[
%	\ord(y_{i_0} - x_{i_0}) = \ord( b - a ) > \ord(x_{i_0}) + \ord N.
%	\]
%	%
%	This shows that $\rv_N(y_{i_0}) = \rv_N(x_{i_0})$ and thus $\rv_N(b) \in \sum_{i=0}^d \xi_i$.
%\end{proof}
%
For any subset $A \subseteq \RV_N$ and $\delta \in \Gamma$, say that $\ord(A) \leq \delta$ if $\ord(x) \leq \delta$ holds for all $x \in A$.
Note that by the above lemma, we have $0 \notin \sum_{i =0}^d \xi_i$, if and only if $\ord(\sum_{i = 0}^d \xi_i) \leq \gamma$ for some $\gamma \in \Gamma$ if and only if $\ord(\sum_{i = 0}^d \xi_i)$ is a singleton.

The following lemma is a reformulation of Hensel's lemma in terms of the partial addition on $\RV_N$.
\begin{lemma} \label{le:hensel}
Let $K$ be a Henselian valued field and $f(x) = \sum_{i=0}^d a_i x^i$ a nonzero polynomial with coefficients in $K$.
If for $N$ and $\xi \in \RV_N^{\times}$ one has
\[
\ord\left(\sum_{i = 1}^{d} \rv_N(i a_i) \xi^{i-1}\right) \leq \min_{1 \leq i \leq d} \ord( a_{i} \xi^{i-1}) + \ord N,
\]
and if there exists some $\tilde{\xi} \in \RV_{N^2}$ such that $\rv_{N^2}(\tilde{\xi}) = \xi$ and
\begin{equation} \label{eq:collision}
	0 \in \sum_{i=0}^d \rv_{N^2} (a_i) \tilde{\xi}^i,
\end{equation}
then there exists a unique $x_0 \in K$ with $\rv_N(x_0) = \xi$ and $f(x_0) = 0$.
\end{lemma}
\begin{proof}
Let $b \in K^{\times}$ be such that $\rv_{N^2}(b) = \tilde{\xi}$ and let $i_0 \in \{0,\dots,d\}$ be maximal such that $\ord (a_{i_0} \xi^{i_0})$ is minimal among the $\ord(a_i \xi^i)$ for $i \in  \{0,\dots,d\}$. Note that by Equation (\ref{eq:collision}), we have $i_0 \geq 1$.
Now consider the polynomial $g(y) \in \cO_K[y]$ given by
\[
g(y) = \sum_{i = 0}^d \frac{a_{i} b^i}{a_{i_0} b^{i_0}} y^i
\]
Then our assumptions on $f$ and $\xi$ imply that
\[ \left\{ \begin{array}{ll}
	\ord g'(1) &\leq \ord N, \\
	\ord g(1)  &>  \ord N^2.
\end{array} \right.
\]
By Hensel's lemma \cite[Thm.~7.3]{Eis95}, we find a unique $y_0 \in K$ with $g(y_0) = 0$ and $\ord(1 - y_0) > \ord N$.
Now set $x_0 \coloneqq b y_0$. Then $f(x_0) = 0$ and $\ord(b - x_0) > \ord b +  \ord N$. The latter precisely means that $\rv_N(x_0) = \rv_N(b) = \xi$.
\end{proof}
\begin{remark}
With the notation of the above lemma, let $h(x) = \sum_{i =0}^d b_i x^i$ be any polynomial with $\rv_{N^2}(b_i) = \rv_{N^2}(a_i)$.
Then the lemma also holds for $h(x)$.
In other words, the existence of the Hensel lift of $\xi$ is a condition on $(\xi,\rv_{N^2}(a_0),\dots,\rv_{N^2}(a_d))$ rather than on $(\xi,a_0,\dots,a_d)$.
\end{remark}

\subsection{The language $\LRV$}\label{sec:lanRV}
We work in a multisorted setting, with sorts $(\VF,(\RV_N)_{N \in \NN})$. Recall that $\NN$ stands for the set of positive integers in this paper.
Let $\LRV$ be the language which precisely contains
\begin{enumerate}
	\item the ring language $\Lring = \{ 0,1,+,-,\cdot\}$ on the valued field sort $\VF$,
	\item the language $\{0,1, \cdot, \mid, \oplus \}$ on the leading term sorts $\RV_N$, where $\mid$ and $\oplus$ are a binary and ternary relation symbol, $\cdot$ is a binary function symbol and $0,1$ are constants,
	\item function symbols $\rv_N \colon \VF \to \RV_N$ for all $N \in \NN$,
	\item function symbols $\rv_{N,M} \colon \RV_{M} \to \RV_N$ whenever $N \mid M$,
	\item a relation symbol $P_{N,d}$ on $\RV_N \times \RV_{N^2}^{d+1}$, for each $d,N \in \NN$.
\end{enumerate}
A valued field $K$, with valuation ring $\cO_K$ and maximal ideal $\cM_K$ has a natural $\LRV$-structure, as follows.
\begin{enumerate}
	\item $\VF$ and $\RV_N$ are interpreted as $K$ and $\RV_N$, respectively,
	\item $\rv_N$ is interpreted as the projection $\rv_N \colon K \to \RV_N$,
	\item $\rv_{N,M}$ are interpreted as the projections $\rv_N \colon \RV_M \to \RV_N$,
	\item on $\RV_N$, the symbols $0,1$ are interpreted as $0$ and $\rv_N(1)$, respectively,
	\item the function symbol $\cdot$ on $\RV_N$ is interpreted as the multiplication on $\RV_N^{\times}$, extended by $0 \cdot x =x \cdot 0 = 0 $ for all $x \in \RV_N$,
	\item $\mid$ and $\oplus$ are interpreted as in Definitions \ref{def:divRV} and \ref{def:oplus},
	\item $P_{N,d}(\xi,\zeta_0,\dots,\zeta_d)$ holds in $K$ if and only if for any (all) $a_0,\dots,a_d \in K$ with $\rv_{N^2}(a_i) = \zeta_i$ the conditions of Lemma \ref{le:hensel} hold, with $f(x) = \sum_{i=0}^da_i x^i$ (and this $N$ and $\xi$).
\end{enumerate}
We will write $\RV_{N,K}$ for the interpretation of the sort $\RV_N$ in $K$ when there is risk of confusion between the two. Additionally, we define the following shorthand
\[ \RVprod \coloneqq \RV_{n_1} \times \dots \times \RV_{n_r}, \]
where $\bar{n} = (n_1,\dots,n_r) \in \NN^r$.
\begin{remark}\label{rem:negP}
Note that $P_{N,d}$ can be expressed by an existential formula without $\VF$-quantifiers in $\LRV \setminus \{ P_{N,d}\}_{N,d}$. By adding it as a symbol to our language, we enforce that also its \emph{negation} is existential and without $\VF$-quantifiers (even quantifier-free).
This is crucial to the reduction step (\ref{it:reduction_PNd}) in the proof of Proposition \ref{prop:CDI}.
\end{remark}

\subsection{Model-theoretic conventions} \label{sec:modtheo-conventions}
Throughout the following sections, we work in a fixed language $\lang$, which is any expansion of $\LRV$ by constants, function symbols and relation symbols such that the new function and relation symbols do not involve any $\VF$-variables. We also work with a fixed $\lang$-theory $T$, expanding the $\LRV$-theory $\THeno$ of characteristic zero nontrivial Henselian valued fields equipped with the natural $\LRV$-structure from Section \ref{sec:lanRV}. An $\lang$-formula $\varphi(x)$ is understood to be in a tuple of variables $x = (x_1,\dots,x_m)$, ranging over any cartesian product of sorts. We typically use the letters $\xi,\zeta,\eta,\rho$ for tuples of variables running over $\RVprod$.

Given some collection $\cK$ of models of $T$, a collection $X = (X_K)_{K \in \cK }$ of subsets $X_K \subseteq K^m \times \RVprod$ is called a \emph{definable set} if there exists some $\lang$-formula $\varphi(x)$ such that $\varphi(K) = X_K$ for all $K \in \cK$. A \textit{definable function} $f \colon X \to Y$ between definable sets $X,Y$ is a collection of functions $(f_K \colon X_K \to Y_K)_K$, each with domain $X_K$ such that $\graph(f) = (\graph(f_K))_K$ is a definable set.

We say that two formulas are \emph{equivalent} if and only if they define the same definable sets (this depends on the choice of $\cK$). The collection $\cK$ will usually consist of all models of $T$ (i.e. $\cK$ is elementary). In this case, two formulas determine the same definable set if and only if they are equivalent with respect to $T$. For how to deal with a small set-theoretical issue here (when speaking of all models), see e.g.~the start of section 2.3 of \cite{CL08}.

\begin{definition} \label{def:existent}
An \emph{existential formula} is a formula of the form $\exists y \varphi(x,y)$, where $\varphi(x,y)$ is a quantifier-free formula. A definable set which is given by an existential formula is called an \emph{existentially definable} set.
\end{definition}
\section{Comparison of leading terms to angular components}

In this section we motivate our choice of working with leading terms instead of angular components: the former pass well to subfields, where the latter do not, see Propositions \ref{le:L-sub} and \ref{prop:ac-restriction-failure}. Since our main theme concerns comparison results when passing from a smaller $p$-adic field to a larger one, this clearly motivates the naturality of our set-up.

\subsection{Valued subfields are $\RV$-substructures}

Let $\Lval$ be the one-sorted language $\Lring \cup \{\cO\}$. Any valued field $K$ has a natural $\Lval$-structure where we interpret $\cO$ as the valuation ring $\cO_K$. Note that for two valued fields $K,L$, we have that $K$ is an $\Lval$-substructure of $L$ if and only if $K$ is a subfield of $L$ and $\cO_K = K \cap \cO_L$.

The following proposition is essential to our approach to descent for the $K$-index.
This is our main motivation for working with the leading term structure, rather than with a language with angular components.

\begin{proposition} \label{le:L-sub}
Let $K,L$ be two Henselian valued fields of characteristic zero. If $K$ is an $\Lval$-substructure of $L$, then $K$ is an $\LRV$-substructure of $L$.
\end{proposition}
\begin{proof}
By assumption, $K$ is an $\Lring$-substructure of $L$. Now note that since $\cO_K  = K \cap \cO_L$, we have that $\cM_K = K \cap \cM_L$. Indeed, $\cM_K$ contains $0$ and all nonzero $x \in \cO_K$ for which $x^{-1} \notin \cO_K = \cO_L \cap K$. Hence for all integers $N>0$ we have that $1 + N\cM_K =  K \cap (1 + N \cM_L)$. Thus the inclusion $K \hookrightarrow L$ induces an inclusion of abelian groups $K^{\times}/(1 + N \cM_K) \hookrightarrow L^{\times}/(1 + N \cM_L)$. We can thus identify $\RV_{N,K}$ with a subset of $\RV_{N,L}$, for each $N >0$. Under this identification, the maps $\rv_N$ and $\rv_{N,M}$ on $L$ restrict to those on $K$.

We still need to prove that the relations $\mid, \oplus, P_{N,d}$ on $K$ are the restrictions of those on $L$. For $|$ this is clear. For the relation $\oplus$, we use Lemma \ref{le:RV-add}. This yields that for $\xi_1,\xi_2 \in \RV_{N,K}$ the sum $\xi_1 +_K \xi_2$ is the image under $\rv$ of an open ball $B_K \subseteq K$. Then $\xi_1 +_L \xi_2$ equals $\rv_N (B_L)$, where $B_L \subseteq L$ is an open ball with the same center and valuative radius as $B_K$. The claim now follows from the fact that $B_K = B_L \cap K$.

Finally, consider the predicates $P_{N,d}$. If for some $\xi \in \RV_{N,K}$ and  $\zeta_0,\dots,\zeta_d \in \RV_{N^2,K}$ the condition $P_{N,d}(\xi,\zeta_0,\dots,\zeta_d)$ holds in $K$, then it clearly also holds in $L$. Conversely, suppose that $P_{N,d}(\xi,\zeta_0,\dots,\zeta_d)$ holds in $L$, for certain $\xi \in \RV_{N,K}$ and $\zeta_0,\dots,\zeta_d \in \RV_{N^2,K}$. Then take lifts $a_i \in \cO_K$ of the $\zeta_i$ and set $f(x) \coloneqq \sum_{i =0}^d a_i x^i$. By assumption, $f(x)$ and $\xi$ satisfy the conditions of Lemma \ref{le:hensel}, in $L$. Hence there is a unique $x_0 \in L$ such that $f(x_0) =0$ and $\rv_N(x_0) = \xi$. Suppose that $x_0 \notin K$, then up to passing to a finite field extension $M$ of $K[x_0]$, we can find some $\sigma \in \Gal(M/K)$ such that $\sigma(x_0) \neq x_0$. But then also $f(\sigma(x_0)) = 0$ and $\rv_N(\sigma(x_0)) = \xi \in \RV_{N,K}$, contradicting uniqueness of $x_0$ (in the Henselian valued field $M$). Hence $x_0 \in K$, and $P_{N,d}(\xi,\zeta_0,\dots,\zeta_d)$ also holds in $K$.
\end{proof}
\subsection{Angular component maps do not always restrict}\label{sec:acN}
Let $K$ be a valued field and recall that for any integer $N>0$ we write $R_N$ for the residue ring $\cO_{K}/(N\cM_{K}	)$.
An \emph{angular component map} $\ac_N \colon K \to R_N$ is a multiplicative homomorphism $K^{\times} \to R_N^{\times}$ such that its restriction to $\cO_K^{\times}$ is the projection onto $R_N^{\times}$, extended by $\ac_N(0) = 0$.
A family of such maps $\{\ac_N \colon K \to R_N \}_{N \in \NN}$ is a called a \emph{compatible system} if they commute with the natural projections $R_{NM} \to R_N$, for all $N,M \in \NN$.

The analogue of Proposition \ref{le:L-sub} fails in any language which includes function symbols $\{\ac_N\}_{N \in \NN}$ for a compatible system of angular component maps: for a field extension $L \geq K$, it is not always possible to find a compatible system of angular component maps on $L$ such that their restrictions to $K$ land in $\cO_K/(N \cM_{K}) \subseteq \cO_L/(N \cM_L)$.

In particular, the language used by Pas in \cite{Pas90} is not suitable for our application in Section \ref{sec:p-E-int}.
Indeed, our approach needs Proposition \ref{le:L-sub} to compare the values of $p$-adic integrals between a field and its finite extensions.

\begin{lemma} \label{le:acN=pi}
Let $K$ be a $p$-adic field (i.e. a finite extension of $\QQ_p$), and let $L$ be a totally ramified finite extension of $K$.
Then the following are equivalent.
\begin{enumerate}
	\item\label{it:acN}  There exists a compatible systems of angular component maps on $L$ whose restrictions to $K$ determine a compatible system of angular component maps on $K$.
	\item There exists a uniformizer $\tau$ of $L$ such that $\tau^{[L:K]} \in K$. \label{it:unif}
\end{enumerate}
\end{lemma}
\begin{proof}
This follows from the fact that $\{\ord N\}_{N \in \NN}$ is cofinal in the value groups of $K$ and $L$, combined with the fact that $K,L$ are complete.
This implies that for a compatible system of angular component maps, we find a unique uniformizer $\tau \in L$ (resp. $\pi \in K$) such that $\ac_N(\tau) = 1$ (resp. $\ac_N(\pi) = 1$) for all $N \in \NN$, and vice versa.
%Conversely, any uniformizer $\tau \in L$ (resp. $\pi \in K$), uniquely determines a system of angular component maps by the conditions $\ac_N(\tau) = 1$ (resp. $\ac_N(\pi) = 1$).
%Hence, to find angular component maps on $L$ such that their restrictions to $K$ take values in $\cO_K/(N\cM_{K}) \subseteq \cO_L/(N \cM_L)$, it is necessary and sufficient that there exists a uniformizer $\tau \in L$ with $\tau^{[L:K]} \in K$.
%For the implication (\ref{it:unif})${\implies}$(\ref{it:acN}), we equip $L$ with the angular component maps determined by $\ac_N(\tau) = 1$, and $K$ with those sending $\tau^{[K:L]}$ to $1 \in R_N$.
%For the converse direction, we find a unique $\tau \in L$ such that $\ac_N(\tau) = 1$ for all $N \in \NN$, as well as a unique $\pi \in K$ such that $\ac_N(\pi) = 1$ for all $N \in \NN$. Then we necessarily have $\tau^{[L:K]} = \pi$.
\end{proof}
\begin{proposition} \label{prop:ac-restriction-failure}
Let $K$ be a $p$-adic field. Then for each $n \in \NN $, there exists a (totally ramified) field extension $L \geq K$ of degree $pn$ such that no compatible system of angular component maps on $L$ restricts to a system of angular component maps on $K$.
\end{proposition}
\begin{proof}
By Lemma \ref{le:acN=pi}, it suffices to find a totally ramified field extension of $L$ of $K$, of degree $pn$ such that for any uniformizer $\tau \in L$ it holds that $ \tau^{pn} \notin K$. Note that if $ \tau^{pn} \eqqcolon \pi \in K$, then actually $L = K[\tau]$ and $\tau$ is a root of $f(x) = x^{pn} - \pi$. Moreover, $\cO_L = \cO_{K}[\tau]$ and we compute that $\ord( f'(\tau)) = \ord(pn) + (pn-1) \ord \tau $. Hence the discriminant of $L$ is $( (pn)^{pn} \pi^{pn -1})$ (\cite[\S II.6]{Ser95}).
We now construct a totally ramified field extension of degree $pn$ with a different discriminant.
Let $\alpha$ be a root of the Eisenstein polynomial $g(x) = x^{pn} + \pi x + \pi $. We have $\ord(g'(\alpha)) = \ord(\pi)  $, whence $L = K[\alpha]$ is a degree $pn$ extension with discriminant $(\pi^{pn})$.
\end{proof}
\subsection{A lemma on sums over $\RVprod$}
In this section we prove a basic estimate for sums of precise forms (Lemma \ref{le:poly-boundRV}), which will be needed in Section \ref{sec:p-E-int}.
By \cite[Prop.~1.8]{R17} one reduces to proving this lemma in a setting where angular components are available. Then our estimate follows easily from \cite[Cor.~5.2.5]{CH18}. As the details are slightly technical, a reader only interested in the broader picture can safely skip this section on a first read and refer back to the statement of Lemma \ref{le:poly-boundRV} when it is used in the proof of Theorem \ref{th:K-ind-global}.
%Although we work in a language with leading term maps instead of angular component maps, we can still use some tools developed in settings with angular component maps. Essentially, this is because an angular component map $\ac_N$ corresponds to a distinguished section for $\ord \colon \RV_N^{\times} \to \Gamma$. Given a compatible system of angular components maps, this allows one to describe each $\RV_N$ purely in terms of the value group $\Gamma$ and residue rings $R_N$. This is made precise in \cite[Prop.~1.8]{R17}.
We recall the generalized Denef-Pas language $\LgDP$ from \cite{CH18}.

\begin{definition} \label{def:gdp}
Let $\LgDP$ be a language with sorts $(\VF,\{\RF_N\}_{N \in \NN},\VG_{\infty})$.
On the valued field sort $\VF$ and the residue ring sorts $\RF_N$ it is the ring language. On the value group sort $\VG_{\infty}$ it is the language of ordered abelian groups $\Log = \{ 0, + , <\}$ together with a constant symbol $ \infty$.
It further contains function symbols $\ord \colon \VF \to \VG_{\infty}$ and $\ac_N \colon \VF \to \RF_N$ for the valuation and angular component maps, respectively.
\end{definition}

Any valued field $K$ which is endowed with a compatible system of angular component maps is naturally and $\LgDP$-structure.
The sorts $\VF$ and $\RF_N$ are interpreted as the valued field $K$ and residue rings $R_N$ respectively, while $\VG_{\infty}$ is interpreted as $\Gamma_{\infty} \coloneqq \Gamma \cup \{ \infty \}$. Here $ \infty$ is a symbol not contained in $\Gamma$ which is larger than all other elements and satisfies $ \gamma + (\infty) = ( \infty) + \gamma  =  \infty $, for all $\gamma \in \Gamma$. Write $\VG$ for the definable set $\VG_{\infty} \setminus \{  \infty \}$.

When $K$ is a finite extension of $\QQ_p$ (i.e. a $p$-adic field), denote by $q_k$ the cardinality of its residue field and by $e_K$ its ramification index. Identifying the value group of $\QQ_p$ with $\ZZ$, one may identify $\Gamma$ with $\frac{1}{e_K} \ZZ$ in such a way that the valuation on $K$ extends the valuation on $\QQ_p$.
%The maps $\ord$ and $\ac_N$ are interpreted as the given valuation and angular component maps, respectively.
%
\begin{lemma}\label{le:poly-boundRV}
Let $f \colon D \subseteq \RVprod \times \RV^{\times} \to \RV^{\times}$ be definable in $\LRV$.
Let $K$ be a $p$-adic field and consider for all $\xi_0 \in \RV^{\times}$ the function
\[ f_K(\cdot,\xi_0) \colon D_K(\xi_0) \subseteq \RVprod \to \Gamma \colon \zeta \mapsto f_K(\zeta,\xi_0). \]
Assume that for all $\xi_0 \in \RV^{\times}$ it holds that
\begin{enumerate}
	\item $f_K(\cdot,\xi_0)$ has finite fibers,
	\item $\ord \left( f_K(D_K(\xi_0),\xi_0) \right)$ is bounded below.
\end{enumerate}
Then there exists a polynomial $p_K(\gamma) \in \QQ[\gamma]$ such that for all $\xi_0 \in \RV^{\times}$ for which $D_K(\xi_0)$ is nonempty it holds that
\[ \sum_{ \zeta \in D_K(\xi_0)} q_K^{-e_K \ord f_K(\zeta,\xi_0)} \leq p_K(\ord \xi_0) q_K^{-e_K \min_{\zeta}(\ord f_K(\zeta,\xi_0)) }. \]
\end{lemma}
\begin{proof}
Choose a compatible system of angular component maps on $K$ and consider the corresponding $\LgDP$-structure on $K$.
Recall that an angular component map $\ac_N$ determines an isomorphism $ u \colon \RV_N^{\times} \to R_N^{\times} \times \Gamma$, which we extend by $u(0) = (0,\infty)$. We can further extend it to any cartesian product of $K$'s and $\RV_N$'s, by setting it to be the identity on all factors $K$. Now define $g_K \coloneqq \ord \circ f_K \circ u^{-1}$. By (the proof of) \cite[Prop.~1.8]{R17}, there exists a definable function $g$ in a language $\lang^{\ac,e}$ such that $g_K$ is its interpretation in $K$.
Parse through \cite[Def 1.7 and p.6]{R17} to see that $\lang^{\ac,e}$ is a definitional expansion of $\LgDP$. Hence, we  may view $g_K$ as the interpretation of an $\LgDP$-definable map
%
%RF_{n_i} instead of just the units, since we want to include zeros in the domain of f
\begin{equation*}
	g \colon \dom(g) \subseteq \left(\VG_{\infty}^r \times  \prod_{i = 1}^r \RF_{n_i}  \right) \times (\RF_1^{\times} \times \VG) \to \VG.
	% &(\zeta,k,\eta,\ell) \mapsto g(\zeta,k,\eta,\ell).
\end{equation*}
By assumption, the maps
\[g_K(\cdot,\eta,\rho,\gamma) \colon \dom(g)(\eta,\rho,\gamma)  \subseteq (\Gamma_{\infty}) \to \Gamma \colon \delta \mapsto g(\delta,\eta,\rho,\gamma) \]
have finite fibers and their range is bounded below, for all $\eta \in \prod_{i} R_{n_i}$, $\rho \in R_1$ and $\gamma \in \Gamma$.
%for all $\zeta \in \prod_{i = 1}^r R_{n_i}^{\times}$, $\eta \in R_1^{\times}$ and $\ell \in \Gamma$.
Now let $p_K(\gamma)$ be the sum of the finitely many polynomials $p_{K,\eta,\rho}(\gamma)$ produced by Lemma \ref{le:poly-bound}.
\end{proof}
\begin{lemma} \label{le:poly-bound}
Let $U,D$ and $f \colon D \subseteq \VG_{\infty}^r \times \VG \times U \to \VG$ be $\LgDP$-definable. Let $K$ be a $p$-adic field and consider for each $(\gamma,u) \in \Gamma \times U_K$ the function
\[ f_K(\cdot,\gamma,u) \colon (\Gamma_{\infty})^r \to \Gamma \colon \delta \to f_K(\delta,\gamma,u).  \]
Suppose that for each $(\gamma,u) \in \Gamma \times U_K$ it holds that
\begin{enumerate}
	\item $f_K(\cdot,\gamma,u)$ has finite fibers,
	\item $f_K(D_K(\gamma,u),\gamma,u)$ is bounded below.
\end{enumerate}
Then for each $u \in U_K$ there exists a polynomial $p_{K,u}(\gamma) \in \QQ[\gamma]$ such that for all $\gamma \in \Gamma$ for which $D_K(\gamma,u)$ is nonempty it holds that
\[
\sum_{\delta \in D_K(\gamma,u)} q_K^{-e_K f_K(\delta,\gamma,u)} \leq p_{K,u}(\gamma) q_K^{-e_K \min_{\delta}(f_K(\delta,\gamma,u))}.
\]
\end{lemma}
\begin{proof}
Write $e = e_K$, $q = q_K$ and rewrite the given sum as
\[
\sum_{ \varepsilon \in f_K(D_K(\gamma,u))} \# \left( f_K(\cdot,\gamma,u)^{-1}(\varepsilon) \right) q^{-e \varepsilon}.
\]
By \cite[Cor.\,5.2.5]{CH18}%
%\footnote{Apply this Corollary with respect to the theory $\THeno$ extended by $T_{\Log}(\ZZ)$ on $\VG$, noting that $\Gamma \cong \ZZ$.} %
we find that after partitioning the domain $D$ of $f$ into finitely many pieces, the definable function sending $(\gamma,u,\varepsilon) \in \Gamma \times U_K \times \Gamma$ to the minimal element (resp. maximal element smaller than $ + \infty$) of $f_K(\cdot,\gamma,u)^{-1}(\varepsilon)$ over all coordinates is bounded below (resp. above) by a function that is linear in $\gamma$ and $\varepsilon$. It follows that for each $u \in U_K$ there is a polynomial $h_u(\gamma,\varepsilon) \in \QQ[\gamma,\varepsilon]$ such that the given sum is bounded above by
\[  \sum_{\varepsilon \in f_K(D_K(\gamma,u)) }
h_{u}(\gamma,\varepsilon) q^{-e \varepsilon} . \]
%f
Thus we are summing (derivatives of) geometric series in $q$. Since $q > 1$, these converge and the lemma follows from the formulas for summing such series (see e.g. \cite[Lemma\,4.4.3]{CL08}).
\end{proof}
%
%
%\begin{example}
%Consider, for example $K = \QQ_3$ and $L = \QQ_3[\alpha]$, where $\alpha$ is a root of the Eisenstein polynomial $f(x) = x^3 +  3x + 3$.
%\end{example}
%
\section{$\exists$-simple formulas}
Let $\lang$ be as in Section \ref{sec:conventions} and recall that it is an expansion of $\LRV$. In particular, $\lang$ includes symbols for partial addition $\oplus$ and the Hensel lift predicates $P_{N,d}$. We refine the strategy of Denef and Pas \cite{Den86,Pas89} to prove a new cell decomposition statement, with extra control on quantifiers. From this refined cell decomposition result, our applications will follow. We first introduce a refined variant of Denef's and Pas's notion of simple formulas (which goes back to a notion by Cohen \cite{Cohen}).
\begin{definition} \label{def:simple}
An $\lang$-formula $\varphi(x)$ is called \emph{$\exists$-simple} if it is an existential formula without quantifiers over the valued field. This means that there exists some quantifier-free $\lang$-formula $\psi(x,\xi)$ such that $\varphi(x)$ equals
\[ (\exists \xi \in \RVprod) \psi(x, \xi) \]
for some $\bar{n} \in \NN^r$ (following the notation of Section \ref{sec:modtheo-conventions}).
\end{definition}
\begin{definition}
If $X$ is a definable set such that there exists an $\exists$-simple formula $\varphi(x)$ defining $X$, then we call $X$ an \emph{$\exists$-simple set}.
\end{definition}
\begin{lemma} \label{le:boolean_ops}
Let $X,Y \subseteq \VF^m \times \RVprod$ be $\exists$-simple sets.
Then both $X \cap Y$ and $X \cup Y$ are $\exists$-simple.
\end{lemma}
%
%\begin{proof}
%Let $\varphi(x,u)$ and $\psi(x,v)$ be quantifier-free formulas, such that $\exists u \varphi(x,u)$ defines $X$ and $\exists v\psi(x,v)$ defines $Y$.
%We may assume that the (bound) variables in the tuple $u$ do not occur in the tuple $v$.
%Then $X \cap Y$ is defined by $\exists u,v (\varphi(x,u) \land \psi(x,v))$. A similar $\exists$-simple formula defines $X \cup Y$.
%\end{proof}
%
\begin{remark} \label{re:fdom}
Let $f \colon X \to Y$ be a definable function whose graph is given by a formula $\varphi(x,y)$. Recall from Section \ref{sec:modtheo-conventions} that we require $X$ to be precisely the domain of $f$, or in other words,
\[ x \in X \leftrightarrow (\exists y \in Y) \varphi(x,y). \]
This may seem a trivial note but is important to keep in mind, see e.g.~Remark \ref{re:Edom} below.
\end{remark}
\begin{definition} \label{def:f-substitution}
Let $f \colon X \to Y$ be a definable function, defined by some formula $\varphi(x,y)$. For any formula $\psi(y,z)$, write
\[ \psi( f(x), z ) \]
as a shorthand for
\[  \exists y (\varphi(x,y) \land \psi(y,z)). \]
\end{definition}
A priori it is not clear if a definable function with an $\exists$-simple graph preserves $\exists$-simple formulas under the kind of substitution considered in Definition \ref{def:f-substitution}.
To this end, we also introduce a corresponding variant on Denef's and Pas's notion of \emph{strongly definable} functions.
It will be a consequence of cell decomposition (Corollary \ref{cor:simple_is_strong}) that each function with an $\exists$-simple graph is already such an $\exists$-strongly definable function.
\begin{definition} \label{def:strongly-def}
Let $X$ be an $\exists$-simple set, $Y$ a definable set and $f \colon X \to Y$ a definable function. The function  $f$ is called \emph{$\exists$-strongly definable} if it preserves $\exists$-simple formulas. That is, for each $\exists$-simple formula $\psi(y,z)$ there exists some $\exists$-simple formula $\varphi(x,z)$ %where $x$ ranges over $X$ and $y$ ranges over $Y$,
such that
\[  \psi(f(x),z) \]
is equivalent to
$ \varphi(x,z).$
\end{definition}
\begin{remark} \label{re:Edom}
If $f \colon X \to Y$ is an $\exists$-strongly definable function and $Z \supseteq X$ is a bigger $\exists$-simple set, then $f$ does not necessarily extend to an $\exists$-strongly definable function $Z \to Y$. The naive procedure of extending by zero may fail if $Z \setminus X$ is not $\exists$-simple.
It is because of such subtleties, that we carefully keep track of the domain of $f$ as in Remark \ref{re:fdom}
\end{remark}
\begin{lemma} \label{le:strongly-def}
Let $X$ be an $\exists$-simple set.
\begin{enumerate}
	\item Let $f \colon X \to Y$ be an $\exists$-strongly definable function and $Z$ an $\exists$-simple subset of $X$. The restriction $f_{|Z} \colon Z \to Y$ is an $\exists$-strongly definable function. \label{it:restr}
	
	\item If $f \colon X \to Y$ and $g \colon Y \to Z$ are $\exists$-strongly definable, then so is $g \circ f$. \label{it:comp}
	
	\item The $\exists$-strongly definable functions $X \to \VF$ form a ring. \label{it:ring}
	
	\item Let $f \colon X \to \RVprod$ be a definable function whose graph is an $\exists$-simple set. Then $f$ is an $\exists$-strongly definable function. \label{it:map-to-aux}
\end{enumerate}
\end{lemma}
\begin{proof}
\begin{enumerate}
	\item Take an $\exists$-simple formula $\varphi(y,u)$ and let $\psi(x,u)$ be an $\exists$-simple formula such that $\varphi(f(x),u)$ is equivalent to $\psi(x,u)$.
	Then we also have that
	\[  (x \in Z \land \varphi(f(x),u)) \leftrightarrow (x \in Z \land \psi(x,u)). \]
	\item Consider an $\exists$-simple formula $\varphi(z,u)$. Then there exists certain $\exists$-simple formulas $\psi(y,u)$ and $\chi(x,u)$ such that
	\[  \varphi(g(f(x)),u) \leftrightarrow \psi(f(x),u) \leftrightarrow \chi(x,u). \]
	\item This follows from (\ref{it:comp}), since addition and multiplication are  $\exists$-strongly definable functions $\VF \times \VF \to \VF$.
	\item Let $\varphi(x,\xi)$ be a simple formula defining the graph of $f$. Consider an arbitrary $\exists$-simple formula $\psi(\xi,y)$. For $x \in X$, the shorthand $\psi(f(x),y)$ stands for
	\[ (\exists \xi \in \RVprod) ( \varphi(x,\xi) \land \psi(\xi,y) ) , \]
	which is already an $\exists$-simple formula as desired. \qedhere
\end{enumerate}
\end{proof}
\begin{lemma} \label{le:conditions}
If $f \colon X \to \VF$ is $\exists$-strongly definable, then so is
\[ X \setminus \{x \mid f(x) = 0\} \to \VF \colon x \mapsto \frac{1}{f(x)}. \]
\end{lemma}
\begin{proof}
By Lemma \ref{le:strongly-def} (\ref{it:restr}) and (\ref{it:comp}), this reduces to checking that the function $\VF \setminus \{0\} \to \VF \colon t \mapsto \frac{1}{t}$ is $\exists$-strongly definable. By the definition of $\exists$-simple formulas, it suffices to show that if $\varphi(t,x,\xi)$ is a quantifier-free formula with $x \in \VF^m$ and $\xi \in \RVprod$, then $\varphi(\frac{1}{t},x,\xi)$ is equivalent to an $\exists$-simple formula.

We observe that the $\exists$-simple formula $\varphi(t,x,\xi)$ does not include any valued field quantifiers, and all $\VF$-terms are polynomials in $t$ and $x$. Additionally, for any $y \in K$, we have $y=0$ if and only if $\rv(y) = 0$. Hence, we may assume that every occurrence of the variable $t$ is inside a term of the form $\rv_N(f(t,x))$, where $f(t,x)$ is a polynomial.

For each such $f(t,x)$, there is some positive integer $d$ such that
\[(\rv_N t)^d \rv_N (f(\frac{1}{t},x)) =  \rv_N g(t,x)\]
for some polynomial $g(t,x)$. Suppose for simplicity that $t$ only occurs in a single term of the form $\rv_N f(t,x)$, then we find some $\exists$-simple formula $\psi(t,x,\xi,\zeta)$ such that $\varphi(\frac{1}{t},x,\xi)$ is equivalent to
\[ \exists \zeta \in \RV_N ( t \neq 0 \land  (\rv_N t)^d \cdot \zeta  = \rv_N g(t,x) \land \psi(t,x,\xi,\zeta) ). \]
For arbitrary $\exists$-simple formulas, we iterate this procedure.
\end{proof}
In the proofs in Section \ref{sec:decomp}, we will sometimes write down formulas that are, strictly speaking, not $\exists$-simple formulas (and not even $\lang$-formulas).
It will be clear from the context that such appearing conditions can equivalently be rewritten as (less transparent) $\exists$-simple formulas.
We illustrate two special cases in the lemma below.
\begin{lemma} \label{le:rewrite}
Let $N, M$ be positive integers. For each of the conditions below there exists an $\exists$-simple $\lang$-formula $\varphi(a,b,c)$ such that, for all valued fields $K \models T$ and all choices of $a,b,c$, that condition is equivalent to $K \models \varphi(a,b,c)$.
\begin{enumerate}
	\item $\ord a > \ord b + \ord c$ for $a,b,c$ in $K$ (or in $\RV_N$),
	\item $a = \rv_{N}(b  + c)$, with $a \in \RV_N$ and $b,c \in RV_{NM}$ (see Remark \ref{re:plus-singleton}).
\end{enumerate}
\end{lemma}
\begin{proof}
We start by considering the first condition, in the case that $a,b,c \in \RV_N$.
Then an equivalent (quantifier-free) formula is the conjunction of $b c | a$ with $ \neg (a | b c )$.
In the case where some of the $a,b,c$ belong to $K$, we first apply $\rv_N(\cdot)$.

The second condition can be split up into two parts. First, it asserts that there is a unique element in $\rv_N(b + c)$, which is equivalent to $\rv_M(b) \neq -\rv_M(c)$. This can in turn be rewritten as $\neg \oplus(\rv_M(b),\rv_M(c),0)$. %technically, this is equivalent to E c' : \oplus(\rv_M c,c',0)
Second, it implies that $a \in \rv_N(b + c)$, which is expressed by the formula
\[ \exists d \in \RV_{NM} ( \oplus(b,c,d) \land \rv_N(d) = a ) . \qedhere \]
\end{proof}
\subsection{Cells}
Our definitions for cells are quite similar to the cells from \cite[Sec.~7]{CL08}. The main difference is that we ask all our data to be $\exists$-simple (and not just definable).
\begin{notation}
Consider definable sets $X,Y,Z$ with $Z \subseteq X \times Y$. For $y \in Y$, write $Z(y)$ for the fiber of $Z$ at $y$, as follows:
\[ Z(y) \coloneqq  \{ x \in X \mid (x,y) \in Z  \} .\]
\end{notation}
\begin{definition}[Cells] \label{def:cells}
Let $U$ be any definable set. An $\exists$-simple set $Z \subseteq \VF \times U$ is called an \emph{$\exists$-simple cell with presentation} $(\lambda,Z_D)$ if the following two conditions hold
\begin{enumerate}
	\item $Z_D \subseteq Z \times \RVprod \subseteq \VF \times U \times \RVprod$ is of the form
	\[ Z_D = \{ (t,x,\xi) \mid (x,\xi) \in D \land \rv_N(t - c(x,\xi)) \in R(\xi) \},\]
	where $D$ and $R$ are $\exists$-simple sets, $c \colon D \to \VF$ is $\exists$-strongly definable and either $\emptyset\not= R(\xi) \subseteq \RV_N^{\times}$ for all $(x,\xi) \in D$, or $R(\xi) = \{0\}$ for all $(x,\xi) \in D$.
	
	\item $\lambda \colon Z \to Z_D$ is an $\exists$-strongly definable bijection, commuting with the projection onto $\VF \times U$.
\end{enumerate}
If $R(\xi) \subseteq \RV_N^{\times}$, then $Z$ is called a \emph{$\exists$-1-cell}, and, if $R(\xi) = \{0\}$, then $Z$ is called a \emph{$\exists$-0-cell}. The function $c$ is called the \emph{center} of the cell and the positive  integer $N$ is called the \emph{depth} of the cell. The $\exists$-simple set $D$ is called the \emph{base} of $Z_D$.
\end{definition}
\begin{remark}
Any $\exists$-simple cell $Z \subseteq \VF \times U$, with presentation $(\lambda,Z_D)$, can be rewritten as a disjoint union of fibers of $Z_D$:
\[ Z = \bigsqcup_{\xi \in \RVprod} Z_D(\xi). \]
Conversely, if $Z \subseteq \VF \times U$ is given as such a disjoint union (of fibers of such $Z_D$), then we can define $\lambda \colon Z \to Z_D$ by asking that $\lambda(t,u)$ is the unique tuple $(t,u,\xi)$ such that $(t,u) \in Z_D(\xi)$. As $\lambda$ has an $\exists$-simple graph and $\xi \in \RVprod$, it is $\exists$-strongly definable, by Lemma~\ref{le:strongly-def} (\ref{it:map-to-aux}). We see that $Z$ is a cell, with presentation $(\lambda,Z_D)$.
\end{remark}
The following theorem will be our main tool for the positive existential uniform $p$-adic integration in Section \ref{sec:p-E-int}; it is our key technical result.
The different notions in this statement are introduced in Section \ref{sec:modtheo-conventions} and Definition \ref{def:cells} (which builds on Definitions \ref{def:simple} and \ref{def:strongly-def})
\begin{theorem}[Cell decomposition] \label{th:CDIII}
Let $Y \subseteq \VF \times (\VF^m \times \RVprod)$ be an existentially definable set. Then $Y$ is the disjoint union of finitely many $\exists$-simple cells.
\end{theorem}

In the next section we will prove Theorem \ref{th:CDIII}, or rather, a special case. In Section \ref{sec:qe}, we will derive quantifier elimination results from that special case of Theorem \ref{th:CDIII} from which the full Theorem \ref{th:CDIII} and an easy description of $\exists$-strongly definable functions will follow, see Theorems \ref{th:E-reductionv1},  \ref{th:E-reduction} and  Corollaries \ref{cor:exists_is_simple}, \ref{cor:simple_is_strong}. In Section \ref{sec:descent} we will then obtain our main goals about descent, again essentially using Theorem \ref{th:CDIII}.

\section{Cell decomposition} \label{sec:decomp}
%
%We essentially prove Proposition~\ref{th:CDIII} by induction on the maximal degree of the polynomials occuring in a formula for $Z$.
To prove Theorem \ref{th:CDIII}, we closely follow and refine Denef's and Pas's strategy of \cite{Den86,Pas90}, the main differences being that we use a language with $\RV$-sorts instead of angular component maps, and, that we finely control quantifiers throughout the whole proof. Indeed, we have to be careful to only introduce existential quantifiers over $\RV_N$ throughout the entire procedure.
%Compare also to \cite{Flen11}, which proves quantifier elimination relative to the $\RV$-sorts, but does give control over the existential quantifiers, nor does the proof include cell decomposition.
%
We inductively prove the following two propositions, for integers $d >0$.
\begin{proposition}[Statement (I)$_d$] \label{prop:CDI}
Let $N$ be a positive integer, $X$ be an $\exists$-simple set, $Z \subseteq \VF \times X$ an $\exists$-simple cell and $f(t,x)$ a polynomial of degree at most $d$ in the $\VF$-variable $t$, whose coefficients are $\exists$-strongly definable functions in $x \in X$.

Then there exist an integer  $q>0$ and a partition of $Z$ into finitely many $\exists$-simple cells, such that on each cell $\tilde{Z} = \bigsqcup_{\xi} \tilde{Z}_D(\xi)$, with center $c(x,\xi)$, the following holds: if we write
\[ f(t,x) = \sum_{i =0}^d a_i(x,\xi) (t-c(x,\xi))^i, \]
then
\[ \rv_N(f(t,x))  = \rv_{N} \left( \sum_{i =0}^d \rv_{N q}( a_i(x,\xi)  (t-c(x,\xi))^i ) \right). \]
\end{proposition}
\begin{remark}
Since the coefficients $a_i(x,\xi)$ are related to the original coefficients of $f(t,x)$ by Euclidean division, they are $\exists$-strongly definable functions on $D$, by Lemma~\ref{le:strongly-def} (\ref{it:ring}).
\end{remark}
\begin{proposition}[Statement (II)$_d$] \label{prop:CDII}
Let $N$, $X$ and $Z$ be as in Proposition \ref{prop:CDI} and let $f_1(t,x),\dots,f_r(t,x)$ be polynomials in $t$ of degree at most $d$, whose coefficients are $\exists$-stronlgy definable functions in $x \in X$. Then there exist an integer $q>0$ and a partition of $Z$ into finitely many $\exists$-simple-cells, such that on each cell $\tilde{Z} = \bigsqcup_{\xi} \tilde{Z}_D(\xi)$, with center $c(x,\xi)$,
\[ \rv_N(f_j(t,x)) = h_j(\rv_{N q }(t - c(x,\xi)), x , \xi) ), \]
for certain $\exists$-strongly definable functions $h_{j}(\zeta, x,\xi)$, for $j = 1,\dots,r$.
\end{proposition}
We say that the $f_j(t,x)$ are \emph{prepared} on the cells in the resulting decomposition from Proposition \ref{prop:CDII}, and, also, that this cell decomposition \emph{prepares} the $f_j(t,x)$.
\begin{remark}
It is implicit in the above statement that the $h_j$ have domain
\[ \dom(h_j) = \{ (t - c(x,\xi), x, \xi ) \mid (t,x,\xi) \in \tilde{Z}_D \} .\]
This set is $\exists$-simple because $\tilde{Z}_D$ is $\exists$-simple and $c(x,\xi)$ is $\exists$-strongly definable.
\end{remark}
Throughout the proofs below, we will refer to $\exists$-simple cells simply as cells.
\begin{proof}[Proof of Statement \textup{(I)}$_d$, assuming \textup{(I)}$_{d-1}$ and \textup{(II)}$_{d-1}$]
We may assume that $d \geq 2$ as the case $d =1$ is straightforward.
We apply the induction hypothesis (I)$_{d-1}$ to the derivative of $f(t,x)$ (with respect to $t$). Up to replacing $Z$ by one of the resulting $\tilde{Z}$, we may find some $q_0 \in \NN$ such that on $Z$
\[ \rv_N(f'(t,x)) = \rv_N\left( \sum_{i=1}^d \rv_{N q_0}\left( i a_i(x,\xi) (t - c(x,\xi))^{i-1} \right) \right). \]
In particular, the right-hand side is a singleton and thus
\begin{equation} \label{eq:derivative}
\ord f'(t,x) \leq \min_{i} \ord \left( i a_i(x,\xi) (t-c(x,\xi))^{i-1}  \right)  + \ord q_{0} .
\end{equation}
Let $M =d! q_{0}$ and note that Equation (\ref{eq:derivative}) implies that the first condition of Lemma \ref{le:hensel} is fulfilled (with $M$ instead of $N$). We proceed with several reduction steps.
\begin{enumerate}
	\item We may assume that $Z$ is a $\exists$-1-cell, since else $f(t,x) = a_0(x,\xi)$.
	\item \label{it:hide_reparam} Suppose that $Z$ has presentation $(\lambda,Z_D)$. Note that a partition of $Z_D$ into cells yields a corresponding partition of $Z$. Hence, up to precomposing with $\lambda^{-1}$, we may as well assume that the coefficients of $f(t,x)$ have domain $D$ and that $Z = Z_D$. Effectively, we may omit the reparametrizing variables $\xi$ from our notation.
	\item Since a cell decomposition for $g(y) = \sum_{i=0}^d a_i(x) y^i$ (say with center $\tilde{c}(x,\zeta)$) yields a cell decomposition for $f(t,x)$ (with center $c(x) + \tilde{c}(x,\zeta)$), we may assume that $c(x)$ is identically zero on $D$.
	\item Up to possibly increasing $q_0$, we may assume that $M$ is a multiple of the depth of $Z_D$. Then replace $\lambda$ by the map $(t,x) \mapsto (\lambda(t,x), \rv_M(t))$ and perform the reduction step (\ref{it:hide_reparam}) again, to reduce to the case where $Z$ is of the form
	\[ Z = \{ (t,x) \mid x \in D \land \rv_{M}(t) = \xi(x)\}, \]
	where $\xi(x)$ is simply the projection onto the last coordinate of $X$ (in particular, it is $\exists$-strongly definable). Note that $\xi(x) \in \RV_N^{\times}$ for all $x$, since $Z$ is an $\exists$-1-cell.
	%
%	\item Up to partitioning $D$, we may assume that each $a_i(x)$ is either identically zero, or never zero. Denote by $I \subseteq \{0,\dots,d\}$ the set of nonzero $a_i(x)$.
%	%
%	\item After further partitioning $D$, we may additionally assume that the order type of $\{ \ord( a_{i}(x) t^i) \}_{i \in I}$ is independent of $x \in D$.
%	Indeed an order type can be expressed by a suitable conjunction over formulas expressing that
%	%
%	\[ (i - j) \ord \xi(x) \leq \ord a_j(x) - \ord a_i(x), \]
%	%
%	for $i,j \in I$.
	%
	\item \label{it:reduction_PNd} We may assume that the condition
	\[ P_{M,d}(\xi(x),\rv_{M^2}(a_0(x)),\dots,\rv_{M^2} (a_d(x)) ) \]
 	holds on all of $D$; indeed, by Equation (\ref{eq:derivative}), its negation implies that
	\[\ord f(t,x) \leq \min_i\left( \ord(a_{i}(x) t^i )\right) + \ord(M^2).\]
	This gives the desired conclusion (with $q = M^2$). Note that the part of $D$ where $P_{N,d}$ does not hold is $\exists$-simple, see Remark \ref{rem:negP}.
%	\item Let  $i_0 \in I$ be maximal, such that $\ord(a_{i_0}(x) t^{i_0})$ is minimal among $\ord (a_i(x) t^i )$, for $i \in I$. We may assume that there is a fixed set $I_0 \subseteq I$ containing those $i$ for which
%	%
%	\[ \ord(a_i(x) t^i ) \leq \ord( a_{i_0}(x)  t^{i_0}) + \ord( N) .\]
	%
%	\item Up to dividing $f(t,x)$ by the unit $a_{i_0}(x)$, we may (using Lemma~\ref{le:conditions}) even assume that $a_{i_0}(x) = 1$ and that $\ord a_i(x) \geq 0$, for all $i \in I$.
\end{enumerate}
%
%If $I_0 = \{i_0\}$, then we have $\rv_{N} f(t,x) = \rv_{N}(t^{i_0})$, and we are done.
%\todoi{The above is perhaps not necessary anymore in the current proof}
%
By this last reduction step, we may assume (by Lemma \ref{le:hensel}) that there exists a definable $d \colon D \to \VF$, such that $f(d(x),x) = 0$ and $\rv_M(d(x)) = \xi(x)$, for all $x \in D$. We would like to use $d(x)$ as our new cell center, since we have that
\[ Z = \{ (t,x) \mid  x \in D \land \ord(t - d(x) ) > \ord d(x) + \ord M  \}. \]
Indeed, the condition $\ord(t-d(x)) > \ord d(x) + \ord M$ is equivalent to $\rv_M(t) = \rv_M(d(x))$ and $\rv_M(d(x)) = \xi(x)$ by construction.
As the above description of $Z$ only depends on $x$ and $\rv(t - d(x))$, it is indeed a cell with center $d(x)$.

Now use $\ord(t-d(x)) > \ord d(x) + \ord M$ and Equation (\ref{eq:derivative}) to see that for all $2 \leq k  \leq d$
\begin{align}
	\ord\left(  \frac{f^{(k)}(d(x),x) }{k!} (t - d(x))^k \right) 	&> \min_{i}( \ord a_{i}(x) d(x)^{ i - 1 } (t-d(x)) ) + \ord M, \notag \\
																	&\geq \ord(f'(d(x),x) (t - d(x))). \label{eq:derivative-approximates}
%	\ord \left( \binom{k}{i} a_k(x) d(x)^{k-i} \right) \geq \min_{ 0 \leq i \leq d} \ord  (a_{i_0}(x)  (k-i) \ord d(x)).
\end{align}
Hence, after Taylor expanding $f(t,x)$ around $d(x)$, it follows that
\begin{equation}
 	\rv_N(f(t,x)) = \sum_{k = 1}^d \rv_N\left(\frac{f^{(k)}(d(x),x)}{k!} (t-d(x))^k \right).
\end{equation}
This is our desired conclusion, with $q = 1$.

So all that is left is to prove that the (definable) function $d(x)$ is $\exists$-strongly definable.
Let $\varphi(t,y,\xi)$ be an $\exists$-simple formula, where $y$ ranges over $\VF^m$ and $\xi$ over $\RV_{ \bar{n} }$.
%Since the language on $\VF$ is an expansion of the language of rings by constants, each occurrence of the variable $t$ is (without loss of generality) inside a term of the form $\rv_N(h_i(t,y))$, for some list of polynomials $h_1(t,y),\dots,h_r(t,y)$.
Without loss of generality, each occurrence of the variable $t$ is inside a term of the form $\rv_N(h_i(t,y))$, for some list of polynomials $h_1(t,y),\dots,h_r(t,y)$.
By Euclidean division, we can write each $h_i(t,y)$ as $f(t,x) q_i(t,x,y) + p_i(t,x,y)  $, where $q_i$ and $p_i$ are polynomials in $t$, with $\exists$-strongly definable functions in $(x,y) \in D \times \VF^m$ as coefficients.
As $f(d(x),x) = 0$, we may replace each occurrence of $h_i(d(x),x,y)$ in $\varphi(d(x),x,\xi)$ by $p_i(d(x),x,y)$.

Since the $t$-degree of the $p_i(t,x,y)$ is strictly smaller than that of $f(t,x)$, we may apply (II)$_{d-1}$ to the $p_i(t,x,y)$. Consequently, we may further rewrite $\varphi(d(x),x,\xi)$ as a finite disjunction over formulas of the form
\[ \begin{array}{ll}
	\exists \zeta \in \RV_{\bar{n}'} (	& (x,y,\zeta) \in C \land \rv_{N_1}(d(x) - E(x,y,\zeta)) \in R(\zeta) \\
										&\land \psi( \rv_{N_1}( d(x) - E(x,y,\zeta) ),y,\xi,\zeta) ).	\end{array}\]
Here $N_1$ is a postive integer, $C,R$ are $\exists$-simple sets, $E(x,y,\zeta)$ is an $\exists$-strongly definable function and $\psi$ is an $\exists$-simple formula.
Omitting the variables of the functions for notational ease, it thus suffices to prove that the function $\rv_{N_1}(d - E)$ is $\exists$-strongly definable.

We first show that a condition of the form $\rv_{N_2}(d) = \eta$ can be by expressed by $\exists$-simple formulas, for any $N_2 \in \NN$.
We claim that such a condition is equivalent to the existence of some $\rho \in \RV_{M^2 N_2}$ such that
\[ \rv_{M}(\rho) = \xi(x) \quad \text{ and } \quad 0 \in \sum_{i = 0}^d \rv_{M^2 N_2}(a_i) \rho^i  \quad \text{ and } \quad \rv_{ N_2}(\rho) = \eta  .\]
Indeed, for one implication, we just take $\delta = \rv_{M^2 N_2}(d(x))$. For the other direction, we notice that the above conditions imply that as in Lemma \ref{le:hensel} $\rv_{M N_2}(\rho)$ lifts uniquely to a zero of $f(t,x)$ (namely $d(x)$).

Up to subdividing $C$, we may assume that exactly one of the two following $\exists$-simple conditions holds on all of $C$
\begin{enumerate}
	\item $ \rv_{M N_1}(d) \neq \rv_{M N_1}(E)$. In this case, we have $\rv_{N_1}(d-E) = \rv_{N_1}( \rv_{M N_1^2}(d) - \rv_{M N_1^2}(E) )$. This function is $\exists$-strongly definable, by Lemmas \ref{le:rewrite} and \ref{le:strongly-def}.
	\item $\rv_{M N_1}(d) = \rv_{M N_1}(E)$. In this case, we have $\ord(d-E) > \ord d + \ord(M N_1)$. In particular, $(E,x) \in Z$. By a similar argument as for Equation (\ref{eq:derivative-approximates}), we find that for all $2 \leq k \leq d$
	\[ \ord\left( \frac{f^{(k)}(E,x)}{k!} (E-d)^{k} \right) >  \ord (f'(E,x) (E - d)) + \ord N_1. \]
%	%
%	\begin{align*}
%		\ord\left( \frac{f^{(j)}(E,x)}{j!} (E-d)^{j} \right) 	&\geq \ord a_{i_0} E^{i_0 - j} + j \ord(E - d), \\
%																&> \ord(a_{i_0} E^{i_0-1}) + \ord(M N_1) + \ord(E- d), \\
%																&\geq \ord (f'(E,x) (E - d)) + \ord N_1,
%	\end{align*}
%	where in the last line we used that $\ord(f'(t,x)) \leq \ord(a_{i_0} t^{i_0 -1}) + \ord M$ on $Z$.
%	Thus after Taylor expanding $f(d,x) = 0$ around $E$, we see that
	%
	Thus, after Taylor expanding $f(d,x)  = 0$ around $E$, we find that
	\[  \rv_{N_1}(-f(E,x)) = \rv_{N_1}(f'(E,x) (E - d) ). \qedhere \]
%	%
%	By bringing $\rv_N(f'(E,x))$ to the other side, we see that $\rv_N(d-E)$ is $\exists$-strongly definable.
\end{enumerate}
\end{proof}
\begin{proof}[Proof of Statement \textup{(II)}$_d$ assuming \textup{(I)}$_d$ and \textup{(II)}$_{d-1}$]
We argue by induction on $r \geq 2$. Let $Z_1,Z_2$ be cells with respective presentations $(\lambda_i,Z_{D_i})$ and centers $c_i$ for $i = 1,2$ such that $Z_1$ prepares $f_1(t,x),\dots,f_{r-1}(t,x)$ and $Z_2$ prepares $f_r(t,x)$.
We partition the intersection $Z_1 \cap Z_2$ into smaller cells, until all $f_i(t,x)$ are prepared.
Taking some sufficiently large $q \in \NN$ and further reparametrizing $Z_1,Z_2$ by one $\RV_{Nq}$-variable, we may assume that their intersection can be written as
\[  \begin{array}{ll}
	Z_1 \cap Z_2 =  \bigsqcup_{\xi,\zeta} \{ (t,x) \mid 	& (x,\xi) \in D_1 \land (x,\zeta) \in D_2, \\
																	& \rv_{Nq}(t - c_1(x,\xi)) = \xi_1, \\
																	& \rv_{Nq}(t - c_2(x,\zeta)) = \zeta_1 \}. \end{array} \]
Write $D$ for the $\exists$-simple set containing those $(x,\xi,\zeta)$ for which both $(x,\xi) \in D_1$ and $(x,\zeta) \in D_2$ and set $Q = Nq$. By adding one more condition to $D$, we may assume that $c_1(x,\xi) \neq c_2(x,\zeta)$ on $D$.
We now partition $D$ such that exactly one of the two conditions below holds identically on $D$. For notational convenience, we again omit the arguments of our functions.
\begin{enumerate}
	\item $\rv_Q(t-c_2) \neq \rv_Q(c_1 - c_2)$. In this case, we have
	\begin{equation} \label{eq:rvt-c1}
		\rv_Q(t-c_1) = \rv_Q( \rv_{Q^2}(t-c_2) - \rv_{Q^2}(c_1 - c_2) ) .
	\end{equation}
	After reparametrizing by one additional variable $\eta \in \RV_{Q^2}$, we may rewrite this part of $Z_1 \cap Z_2$ as
	\begin{equation*}  \begin{array}{ll}
	 	\bigsqcup_{\xi,\zeta,\eta} \{ (t,x) \mid & (x,\xi,\zeta) \in D, \\
		& \xi_1 = \rv_{Q}(  \eta + \rv_{Q^2}(c_1 - c_2) ), \\
		& \rv_{Q}(\eta)  \neq \rv_Q(c_1 - c_2), \\
		& \rv_{Q^2}(t - c_2) = \eta
 \}, \end{array} \end{equation*}
	which is a cell (note that we implictly use Lemma \ref{le:rewrite}).
	Moreover, by the equation (\ref{eq:rvt-c1}) all $f_i(t,x)$ are prepared on this cell.
	\item $\rv_Q(t-c_2) = \rv_Q(c_1 - c_2)$. Since $c_1 \neq c_2$, condition is equivalent to
	\[ \ord( (t - c_2) - (c_1 - c_2) ) > \ord(c_1 - c_2) + \ord Q. \]
 	Hence, this part of $Z_1 \cap Z_2$ can be rewritten as
	\begin{equation*} \begin{array}{ll}
		\bigsqcup_{\xi,\zeta} \{ (t,x) \mid & (x,\xi,\zeta) \in D, \\
		& \ord \xi_1 > \ord(c_1 - c_2) + \ord Q, \\
		& \rv_Q(c_1 - c_2) = \zeta_1, \\
		& \rv_{Q}(t - c_1) = \xi_1
		\}. \end{array} \end{equation*}
	This is again a cell. Note that it prepares $f_r$ because of the equality $\rv_Q(t-c_2) = \rv_Q(c_1 - c_2)$. \qedhere
\end{enumerate}
\end{proof}
%
%\begin{proposition}[Cell decomposition] \label{th:CDIII}
%Let $Y \subseteq \VF \times (\VF^m \times \RVprod)$ be an $\exists$-simple set. Then $Y$ is the disjoint union of finitely many $\exists$-simple cells.
%\end{proposition}
%
The following Proposition \ref{prop:CDIII} is slightly weaker than Theorem \ref{th:CDIII} as it only applies to $\exists$-simple sets.
However, in the next section we show that actually all existentially definable sets are $\exists$-simple (Corollary \ref{cor:exists_is_simple}) thus completing the proof of Theorem \ref{th:CDIII}.
\begin{proposition} \label{prop:CDIII}
Let $Y \subseteq \VF \times (\VF^m \times \RVprod)$ be an $\exists$-simple set. Then $Y$ is the disjoint union of finitely many $\exists$-simple cells.
\end{proposition}
\begin{proof}
Let $\varphi(t,x,\xi)$ be an $\exists$-simple formula defining $Y$, with $(t,x) \in \VF \times \VF^m$ and $\xi \in \RVprod$.
We may find a finite list of polynomials $f_1(t,x),\dots,f_r(t,x)$ such that each occurrence of $t$ is inside a term of the form $\rv_N(f_i(t,x))$ for some $i \in \{1,\dots,r\}$.
We now apply Proposition \ref{prop:CDII} to the polynomials $f_i(t,x)$. This yields a partition of $\VF \times \VF^m$ into $\exists$-simple cells. After taking a cartesian product with $\RVprod$, we obtain a corresponding decomposition of $\VF \times (\VF^m \times \RVprod)$.

Let $Z = \bigsqcup_{\zeta} Z_D(\zeta)$ be one of the resulting cells, say with center $c(x,\zeta)$. Assume that $Z$ is an $\exists$-1-cell (the case where $Z$ is a $0$-cell is similar, but more straightforward).
We need to show that $Z \cap Y$ can again be written as a cell.
We may assume that $\rv_{N q}(t-c(x,\zeta)) = \zeta_1$ on $Z_D$, possibly after adding one variable to the tuple $\zeta$.
Since $Z$ was obtained from Proposition \ref{prop:CDII}, the restriction of $\rv_N(f_i(t,x))$ to $Z_D$ is an $\exists$-strongly definable function $h_i(\zeta_1,x,\zeta)$. We thus find an $\exists$-simple formula $\psi(x,\xi,\zeta)$  which is equivalent to $\varphi(t,x,\xi)$ whenever $(t,x,\zeta) \in Z_D$
Then $Z \cap Y$ is given by
\[ \bigsqcup_{\zeta}\{ (t,x,\xi) \mid (x,\zeta) \in D \land \psi(x,\zeta,\xi) \land \rv_N(t - c(x,\zeta)) = \zeta_1 \}, \]
and hence it is a cell.
\end{proof}
%
%By allowing additional constants in $\Lvf$, it is possible to prove directly a version of Proposition \ref{prop:CDII} which also holds for, say polynomials $f(t,x) \in \QQ_p[t,x]$. More generally, if $K$ is a field satisfying $\THeno$, and $X \subseteq K^n$ is a subset given by an existential formula, possibly \emph{with parameters}, then we can decompose $X$ into $\exists$-simple cells. Moreover, we can do this uniformly in all field extensions of $K$. More precisely, we remark the following.
%
%\begin{remark}[Cell decomposition with parameters]
%Fix a $p$-adic field $K_0$ and let $\varphi(x,y)$ be an $\exists$-simple $\lang$-formula, in a tuple of variables $(x,y) \in (\VF^{m+1} \times \RVprod) \times \VF^n$.
%For any $a_0 \in K_0^n$ and any $\lang$-embedding $K_0 \subseteq K$, we can consider the sets $\varphi(K,a_0) \coloneqq X_K$.
%
%By cell decomposition, we may find formulas $\psi_1(x,y),\dots,\psi_n(x,y)$ such that each $\psi_i$ defines an $\exists$-simple cell and the resulting cells partition the definable set corresponding to $\varphi(x,y)$. In particular, for any $\lang$-embedding $K_0 \subseteq K$, the sets $\psi_i(K,a_0)$ partition $X_K$.
%\end{remark}
%%
\section{Reduction of existential quantifiers}\label{sec:qe}
We now come to our first application of existential cell decomposition.
The following existential quantifier reduction statement in particular implies that any existentially definable set is already $\exists$-simple (Corollary \ref{cor:exists_is_simple}), thus completing the proof of Theorem \ref{th:CDIII}.
%We formulate the quantifier reduction statement in terms of existential formulas, but an analogous statement for universal formulas follows immediately.
%
\begin{theorem}[$\exists$-VF reduction] \label{th:E-reduction}
Any existential $\lang$-formula is equivalent to an existential $\lang$-formula without any $\VF$-quantifiers.
\end{theorem}
\begin{proof}
Since existential quantifiers commute, it suffices to show that if $\psi(t,x,\xi)$ is a quantifier-free formula, then
$$
(\exists t \in \VF)\psi(t,x,\xi)
$$
 is $\exists$-simple.
Let $Z \subseteq \VF \times (\VF^m \times \RVprod)$ be the definable set given by $\psi$.
We need to show that the projection of $Z$ onto $U$ is $\exists$-simple.
By Proposition~\ref{th:CDIII}, we may reduce to the case where $Z$ is a single $\exists$-simple-cell, say with presentation $(\lambda,Z_D)$.
But then the projection of $Z$ onto $U$ equals the projection of $D$ onto $U$.
The latter is $\exists$-simple because $D \subseteq U \times \RVprod$.
\end{proof}
\begin{proof}[Proof of Theorem \ref{th:E-reductionv1}]
This is the an instance of Theorem \ref{th:E-reduction}, with $\lang = \LRV$ and $T = \THeno$.
\end{proof}
\begin{remark}
The fact that the language on $\VF$ is an extension of the ring language by only constants symbols is crucial for Theorem \ref{th:E-reduction}.
Indeed, consider for example  the language $\lang' = \LRV \cup \{ P_2(x) \}$, where $P_2(x)$ is a new unary predicate on the valued field sort $\VF$. Let $T' \supseteq T$ express that $P_2(x)$ is equivalent to $(\forall y \in \VF)(x \neq y^2)$.
Consider the statement $\exists x P_2(x)$ and suppose it was equivalent to some existential $\lang$'-sentence $\varphi$ without any $\VF$-quantifiers. Then $\varphi$ would be an existential sentence in the language $\lang'' = \LRV \cup \{ P_2(n) \}_{n \in \ZZ}$.
Now let $k$ be an algebraically closed field. Then $\exists x P_2(x)$ holds in the valued field $k((t))$ (with valuation ring $k[[t]]$), but not in its algebraic closure $L$.
But as $k((t))$ is an $\lang''$-substructure of $L$ and $\varphi$ is existential we must necessarily have $L \models \varphi$, contradiction. Note that any $\lang'$-formula is still equivalent modulo $T'$ to an $\lang'$-formula without valued field quantifiers, by the classical quantifier elimination results in the style of Pas, see e.g.~\cite{Flen11}.
\end{remark}
\begin{corollary} \label{cor:exists_is_simple}
Any existentially definable set is $\exists$-simple.
\end{corollary}
\begin{proof}
This follows immediately from Theorem \ref{th:E-reduction}.
\end{proof}
\begin{proof}[Proof of Theorem \ref{th:CDIII}]
Combine Proposition \ref{prop:CDIII} with  Corollary \ref{cor:exists_is_simple}.
\end{proof}
\begin{corollary} \label{cor:simple_is_strong}
Any definable function $f \colon X \to Y$ with an $\exists$-simple graph is $\exists$-strongly definable.
\end{corollary}
\begin{proof}
By Lemma \ref{le:strongly-def} (\ref{it:map-to-aux}) and induction, it suffices to treat the case where $Y = \VF$.
Let $\varphi(x,t)$ be an $\exists$-simple formula defining $\graph(f)$ and let $\psi(t,y)$ be any $\exists$-simple formula.
We need to show that $\psi(f(x),y)$ is equivalent to an $\exists$-simple formula. Recall that $\psi(f(x),y)$ is just a shorthand for
\[ (\exists t \in \VF)( \varphi(x,t) \land \psi(t,y) ). \]
As this is an existential formula, we may conclude by Theorem \ref{th:E-reduction}.
\end{proof}

We also obtain an ``existential AKE-like-principle'', similar to \cite[Thm.~7.1]{AF16} (but relative to $\RV_N$), where AKE stands for Ax-Kochen and Ershov. For any $\LRV$-structure $K$, we write $\Th_{\exists}(K)$ for the set of all existential $\LRV$-sentences which hold in $K$. Let $\LRV_{|\RV}$ be the restriction of $\LRV$ to the sorts $\{\RV_N\}_{N \in \NN}$. Then we similarly let $\Th_{\exists}( \{ \RV_{N,K} \}_{N \in \NN} )$ be the existential theory of the $\LRV_{|\RV}$-structure $\{\RV_{N,K}\}_{N \in \NN}$.

\begin{corollary} \label{cor:E-AKE}
Let $K, L$ be Henselian valued fields of characteristic zero, then
\[ L \models \Th_{\exists}(K) \text{ if and only if } \{ \RV_{N}(L)\}_{N \in \NN} \models \Th_{\exists}( \{\RV_{N}(K)\}_{ N \in \NN} ) .\]
\end{corollary}
\begin{proof}
Assume that $\{ \RV_{N}(L)\}_{N \in \NN} \models \Th_{\exists}( \{\RV_{N}(K)\}_{ N \in \NN} )$. Note that this implies that $K$ and $L$ have the same residue field characteristic $p$. Indeed for each prime $p$ it holds that $K$ has residue characteristic $p$ if $0 \in \sum_{i=1}^p 1_{\RV_K} $ and $K$ has residue characteristic different from $p$ if $\ord( \sum_{i = 1}^p 1_{\RV_{K}} ) = 0$.
Suppose $K$ has residue characteristic $p > 0$ (the case $p = 0$ is similar) and let ${\THen}_{,0,p}$ be the $\LRV$-theory of nontrivial Henselian valued fields of mixed characteristic $(0,p)$.
Now let $\varphi$ be an existential $\LRV$-sentence and suppose that $K \models \varphi$.
We show that it is equivalent modulo ${\THen}_{,0,p}$ to an existential $\LRV_{|\RV}$-formula.
By Theorem \ref{th:E-reduction}, we may assume that $\varphi$ contains no $\VF$-quantifiers.
Then every $\VF$-term in $\varphi$ is given by some $m \in \ZZ$.
Hence, there exists an existential $\LRV_{|\RV}$-formula $\psi(y_1,\dots,y_r)$ and integers $m_i,N_i$ for $i = 1,\dots,r$ such that $\varphi$ is equivalent to $\psi(\rv_{N_1}(m_1),\dots,\rv_{N_r}(m_r) )$. Now take each $i \in \{1,\dots,r\}$ an integer $M_i$ such that $p^{M_i} \geq m_i$. Then we have
\[ \rv_{N_i}(m_i) = \rv_{N_i}\left( \sum_{k = 1}^{m_i} 1_{\RV_{N_i p^{M_i} } }\right). \]
Hence, we may find an existential $\LRV_{|\RV}$-sentence which is equivalent to $\varphi$ modulo ${\THen}_{0,p}$. In particular, $L \models \varphi$.
\end{proof}
\begin{remark}
In \cite{AF16} a result similar to Corollary \ref{cor:E-AKE} is proved for \emph{equicharacteristic} Henselian valued fields (of any characteristic). They work in a three-sorted structure, with sorts for the valued field, residue field and value group. Their results do not extend to the mixed characteristic case, as demonstrated  in \cite[Rem.~7.4]{AF16}. Because we allow data from all $\RV_N$ (and not just $\RV_1$), our results hold without restriction on the residue characteristic or ramification.
We also note that a version of Corollary \ref{cor:E-AKE} trivially holds for fields of positive characteristic. Indeed, if $K$ is of characteristic $p > 0$, then $\RV_{p,K}^{\times} \cong K^{\times}$ and there is an existential $\LRV_{|\RV}$-formula defining the addition of $K$ on $\RV_{p,K}$.
\end{remark}
\section{Descent for the $K$-index}\label{sec:descent}
\subsection{The $K$-index for functions on a non-Archimedean local field $K$}
We now define the $K$-index for functions $F \colon X \subseteq K^n \to \CC$, where $K$ is a non-Archimedean local field. Similar terminology exists for the Archimedean fields $\RR$ and $\CC$ instead of $K$, and is usually called the real, resp.~complex, log canonical threshold, or, the real, resp.~complex Arnold index, see \cite{Arnold.G.V.II}, \cite{Kollar}, \cite{MustJAMS,Mustata2}.
%This value is a measure of how singular $f$ is on its domain.
%For example, if $f(x) = \abs{1/q(x)}_K$, with $q(x) \in K[x]$ a polynomial and $\abs{\cdot}_K$, then by construction $\ind^K_X(f)$ is the $K$-log canonical threshold of the function $q \colon \cO_K \to K \colon x \mapsto q(x)$ (denoted by $\lambda( \mathcal{I}_q )$ in \cite{VZ07}).

First, we fix some notation regarding integrals.
Let $K$ be a non-Archimedean local field $K$, namely, $K$ is a finite extension of $\FF_p((t))$ or $\QQ_p$ of some degree $n$. Write $n = e_K f_K$ where $p^{f_K} \eqqcolon q_K$ is the cardinality of the residue field and $e_K$ is its ramification index.
The valuation $\ord \colon K^\times \to \frac{1}{e_K} \ZZ = \Gamma$ induces a non-Archimedean absolute value $\abs{x}_K = q_K^{- e_K \ord x}$.
We equip $K$ with its additive Haar measure, normalized such that $\cO_K$ has measure $1$.
All integrals will be with respect to this measure.
%Note that with these conventions we have the equality $\int_{a \cO_{K}} \, dx  = \abs{a}_K$ for all $a \in K$.
%
\begin{definition}
	Let $X$ be a %bounded,
measurable subset of $K^n$ and let $F \colon X \to \CC$ be a measurable function. Denote by $\abs{\cdot}_{\CC}$ the absolute value on $\CC$.
	\begin{enumerate}
		\item We define the \emph{$K$-index} $\ind^K_X(F)$ of $F$ as the supremum of all $s \in \RR_{> 0}$ for which
		\[
		\int_X \abs{F(x)}_{\CC}^s \, dx < + \infty
		\]
		if this supremum exists in $\RR_{> 0}$. We set $\ind^K_X(F) = 0$ if the set of such $s$ is empty and we set $ \ind^K_X(F) = + \infty$ if it is unbounded.
		
		\item For $P \in X$, we define $\ind^K_P(F)$, the \emph{$K$-index of $F$ at $P$}, as the supremum of all values
		\[
		\ind^K_{X \cap Y}(F)
		\]
		as $Y$ ranges over all neighborhoods of $P$ in $K^n$.
	\end{enumerate}
\end{definition}
\begin{remark}
Let $g(x) \in \cO_K[x]$ be a nonconstant polynomial with $g(0) = 0$, considered as a map $g \colon \cO_{K}^n \to K$. Suppose that $K$ is of characteristic zero. If $F(x) = \abs{g(x)}^{-1}_K$ almost everywhere, then $\ind^K_{P}(F)$ with $P=\{0\}$  coincides with the value $\lambda( \cI_g )$ defined in \cite{VZ07}. By \cite[Rem.\,2.8]{VZ07} $\lambda( \cI_g )$ is at least the (complex) log canonical threshold of $g$ at zero.
%which is a well-known singularity invariant
%Moreover, if $L$ is a sufficiently large finite field extension of $K$, then the $L$-log canonical threshold of $p$ equals $\lct(p)$.
%A similar comparison holds for the $K$-index at at point $P$ and the local log canonical threshold.
\end{remark}
\subsection{Uniform $p$-adic integration over existentially definable sets} \label{sec:p-E-int}
In this section, we prove that if $K$ is a $p$-adic field, and $f \colon X\subseteq K^n \to K$ is a function whose graph is defined by an existential $\LRV$-formula, then the $K$-index of $x \mapsto \abs{f(x)}_K$ does not increase when passing to a finite field extension (Theorem \ref{th:K-ind-global}). Informally, this expresses that $f(x)$ can not suddenly become less singular by passing to larger fields. This can be considered as a result about descent, and, about semi-continuity, for the $p$-adic indices. If $f(x)$ is not existentially definable, then such form of semi-continuity does not necessarily hold, as illustrated in Example \ref{ex:descent-failure}. These descent results will be extended to local fields of (large) positive charactgeristic in Section \ref{sec:7.4}, by the transfer principle for motivic and uniform $p$-adic integrals.
 %
%\subsection{Uniform $p$-adic integration over existential sets}

The fact that we don't consider linear combinations (and neither differences) of functions of the form $x \mapsto \abs{f(x)}_K$ can be considered as a basic form of positivity (or, non-negativity) of the functions we consider for our study of descent, similar to semi-ring approach (namely without differences) from \cite{CL08}. In fact, we work with some kind of double positivity, namely the mentioned non-negativity of the functions $\abs{f(x)}_K$, together with the existential nature of our objects, as explained in the introduction.

For the purpose of integration, we restrict to a setting where $\lang$ is an expansion of $\LRV$ by constants only and where all the fields under consideration are non-Archimedean local fields. More precisely, for the remainder of this section, we adopt the following conventions.
%We also coarsen our notion of definable sets accordingly (in the notation of Section \ref{sec:modtheo-conventions}, we take a smaller $\cK$). We give a precise description of our notion of definable sets in this setting below.
%
\begin{notation}
Write $\Loc^0$ for the collection of all non-Archimedean local fields of characteristic zero, and, given any $K_0 \in \Loc^0$ write $\Loc^0_{K_0}$ for the collection of all $K \in \Loc^0$ that are finite field extensions of $K_0$.
Say that $X$ is a \emph{definable set} if it is either $\LRV$-definable or $\LRV \cup K_0$-definable, for some $K_0 \in \Loc^0$.
Note that if $X$ is $\LRV \cup K_0$-definable then it has a natural interpretation $X_K$, for each $K \in \Loc^0_{K_0}$.
\end{notation}
%
%
%\begin{definition} \label{def:defK}
%\begin{enumerate}
%	\item A collection $X = (X_K)_K$ of subsets $X_K \subseteq K^m \times \RVprod$, where $K$ runs over $\Loc^0$ is a \emph{definable set}, if there exists an $\LRV$-formula $\varphi(x)$ such that $\varphi(K) = X_K$ for all $K \in \Loc^0$.
%	
%	\item A collection $Y = (Y_K)_K$ where $K$ runs over $\Loc^0_{K_0}$ is called a ($K_0$-)\emph{definable set}, if there exists an $\LRV \cup K_0$-formula $\psi(y)$ such that $Y_K = \psi(K)$, for all $K \in \Loc^0_{K_0}$.
%\end{enumerate}
%\end{definition}
%
%
%\begin{remark}
%The notions of definable functions, $\exists$-simple sets and $\exists$-strongly definable functions carry over to the setting of Definition \ref{def:defK}%
%\footnote{In the context of this section, two formulas are equivalent if they define the same definable sets. This is coarser than $\THen$-equivalence.}.
%By applying Theorem \ref{th:CDIII} to the underlying formula $\varphi(x)$ (resp. $\psi(x,y)$), we also obtain an analogous cell decomposition statement in this setting.
%Also remark that if $(f_K)_K$ is a definable function, defined by an (existential) $\lang$-formula $\varphi(x,y)$, then
%we can partition the definable set $\{(x,y) \mid f(x)  = y \}$ into $\exists$-simple cells by applying Theorem \ref{th:CDIII} to $\varphi(x,y)$.
%\end{remark}
%
%From now on we will use the term \emph{definable set} to refer to any of the two cases in the above definition, and will only use ``$K_0$-definable set'' to emphasize that we allow parameters from $K_0$ in the sort $\VF$.

The following lemma is a direct consequence of cell decomposition (see also \cite{Pas89,Pas90}).
\begin{lemma} \label{le:int}
Let $U$ and $X \subseteq \VF^m \times U$ be existentially definable.
Then, there exist finitely many positive integers $M_1,\dots,M_r$ and existentially definable sets $D_1,\dots,D_r \subseteq \RVprod \times U$ such that for each $K \in \Loc^0$ (resp. $\Loc^0_{K_0}$) and each choice of $u \in U_K$ the following equality holds
\[ \int_{X_K(u)} \, dx =
\sum_{i = 1}^r q^{-(m + e_K \ord M_i)}
\sum_{ \xi \in D_{i,K}(u) }
q_K^{- e_K(\ord(\xi_1) + \dots + \ord(\xi_{m})) },
\]
in the extended real numbers $\RR \cup \{ + \infty \}$.
\end{lemma}
\begin{proof}
We prove this by induction on $m$.
By Theorem \ref{th:CDIII}, we may reduce to the case where $X \subseteq \VF \times (\VF^m \times U)$ is a single $\exists$-1-cell, say with center $c = c(x,u,\xi)$.

By $\sigma$-additivity and Fubini-Tonelli, it follows for all $K \in \Loc^0$ and $u \in U_K$ that
\[
\int_{X_K(u)} \, dx \,dt =
\sum_{ \xi \in \RVprod }
\int_{(u,x,\xi) \in D_K } \left( \int_{ \rv_N(t-c) = \xi_1  } \, dt \right) \, dx .
\]
As the Haar measure is translation invariant, the inner integral evaluates to
\[q_K^{- e_K (\ord(\xi_1) + \ord(N) ) + 1}.\]
Now apply the induction hypothesis to $D \subseteq \VF^m \times (U \times \RVprod)$.
\end{proof}
\begin{notation}
If $f \colon X \subseteq \VF^m \to \VF$ is a definable function and $K \in \Loc^0$ (resp. $\Loc^0_{K_0}$), then we write $\ind^K_X(\abs{f})$ rather than $\ind^K_{X_K}(\abs{f_K}_K)$.
\end{notation}
We now come to our main result: descent for the $K$-index of existentially definable functions.
\begin{theorem} \label{th:K-ind-global}
Let $X \subseteq \VF^m$ be existentially definable and let $f \colon X \to \VF$ be an existentially definable function. If $K \in \Loc^0$ (resp. $\Loc^0_{K_0}$), then we have for all finite field extensions $L \geq K$ that
\[
\ind_X^L( \abs{f} ) \leq \ind_X^K( \abs{f} ).
\]
\end{theorem}
\begin{proof}
Write $e \coloneqq e_K$ as well as $q \coloneqq q_K$.
Let $Y(\xi_0)$ be the definable set given by the condition $x \in X $, together with $\rv(f(x)) = \xi_0 \in \RV^{\times}$.
By $\sigma$-additivity and linearity, it follows for each $s \in \RR$ that
\begin{equation} \label{eq:int_fs}
	\int_{X_K} \abs{f_K(x)}_K^s \, dx = \sum_{\xi_0 \in RV}  q^{- e \ord(\xi_0) s}  \int_{ Y_K(\xi_0) } \, dx,
\end{equation}
where the left-hand side is finite if and only if the right-hand side is.

Setting $a_K(\xi_0) = \int_{ Y_K(\xi_0) } \, dx$, we get
\[
\int_{X_K} \abs{f_K(x)}_K^s \, dx =   \sum_{\ord \xi_0 < 0} a_K(\xi_0) (q^s)^{- e \ord \xi_0 } + \sum_{\ord \xi_0 \geq 0} a_K(\xi_0) (q^s)^{ - e \ord \xi_0}  .
\]
First consider the case where all $a_K(\xi_0)$ are finite. Since they are non-negative, it follows that if the second summand converges for some $s = s_ 0$, then it will converge for all $s \geq s_0$.
Now view the first summand as power series in a variable $q^s$ and write $q^{\lambda_K}$ for its radius of convergence.
By definition of the $K$-index, this means that $\lambda_K = \ind^K_X(\abs{f})$.
We may compute this radius using the formula
\[
q^{- \lambda_K} = 	
	\limsup_{ n \to \infty }
	\left( \sum_{-e \ord \xi_0 = n} a_K(\xi_0) \right)^{ \frac{1}{n} }.
\]
Using Lemma \ref{le:int}, we can compute this as a maximum over upper limits of the form
\begin{equation} \label{eq:limsup_D}
	\limsup_{n \to + \infty} \left(
	\sum_{- e \ord \xi_0 = n} \sum_{\zeta \in D_K(\xi_0)}  q^{- e(\ord \zeta_1 + \dots + \ord \zeta_m) }
	\right)^{ \frac{1}{n} },
\end{equation}
where $D$ is existentially definable.
We claim that this value only depends on the function
\[
m_{D,K}(\xi_0) \coloneqq \min\{ \ord(\zeta_1 \cdots \zeta_m) \mid \zeta \in D_K(\xi_0)\}.
\]
More precisely, we claim that (\ref{eq:limsup_D}) can be computed as
\begin{equation} \label{eq:claim}
	\limsup_{\ord \xi_0 \to - \infty} \left( q^{ \frac{ m_{D,K}(\xi_0)}{ \ord \xi_0} } \right).
\end{equation}
Indeed, since $a_K(\xi_0)$ is finite for all $\xi_0$, then the interpretation $\theta_K$ of the (existentially) definable map
\begin{equation} \label{eq:theta}
	\theta \colon D \subseteq \RVprod \times \RV^{\times} \to \RV \colon (\zeta,\xi_0) \mapsto \prod_{i = 1}^m \zeta_i
\end{equation}
satisfies the hypotheses of Lemma \ref{le:poly-boundRV}. It follows that there exists some polynomial $p_K(\gamma) \in \QQ[\gamma]$ such that
\[
q^{-e m_{D,K}(\xi_0) }
\leq \sum_{\zeta \in D_K(\xi_0) }  q^{- e \ord(\zeta_1 \cdots \zeta_m)  }
\leq p_{K}(\ord(\xi_0)) q^{-e m_{D,K}(\xi_0)},
\]
The claim now follows, since $(p_{K}(\ord \xi_0))^{\frac{1}{- e \ord \xi_0}}$ tends to $1$ as $\ord \xi_0$ tends to $- \infty$.

The theorem now follows from the claim and the fact that formula (\ref{eq:claim}) also applies to the computation of $\lambda_L  = \ind^L_X(\abs{f})$, (up to replacing e.g. $D_K$ by $D_L$, but without changing $D$). Since $D$ is existentially definable, we have $D_K \subseteq D_L$. It then follows that $m_{D,K}(\xi_0) \geq m_{D,L}(\xi_0)$ for all $\xi_0$, whence $q^{- \lambda_K} \leq q^{- \lambda_L}$ and thus $\ind^K_X(\abs{f}) \geq \ind^L_X(\abs{f})$.

Finally, consider the case where $a_K(\xi_0)$ is infinite for some $\xi_0 \in \RV_K$. Let $D_1,\dots,D_r \subseteq \RVprod \times \RV$ be the existentially definable sets produced by Lemma \ref{le:int}. Then by Lemma \ref{le:poly-boundRV} there must be some $i \in \{1,\dots,r\}$ such that
\[\{ \ord(\xi_1 \dots \xi_m) \mid (\xi_1,\dots,\xi_m) \in D_{i,K}(\xi_0)) \} \subseteq \Gamma \]
is not bounded below, or such that there exists some $\gamma \in \Gamma$ such that $\ord(\xi_1 \dots \xi_m) = \gamma$ for infinitely many tuples $(\xi_1,\dots,\xi_m) \in D_{i,K}(\xi_0)$. Since $D_{i}$ is existentially definable, it follows that $D_{i,K}(\xi_0) \subseteq D_{i,L}(\xi_0)$, whence also $a_L(\xi_0)$ is infinite.
We conclude that in this case $ \ind^K_X(f) = \ind^L_X(f) = 0$.
\end{proof}
\begin{proof}[Proof of Theorem \ref{th:intro-descent}]
This is included in Theorem \ref{th:K-ind-global}, as any $\Lval$-definable set is also $\LRV$-definable.
\end{proof}

We have also a local form of descent, at a point $P$.

\begin{corollary} \label{cor:K-ind-local}
Let $f \colon X \subseteq \VF^m \to \VF$ be existentially definable. Let $K \in \Loc^0$ and $P \in X_K$. For any finite field extension $L \geq K$ it holds that
\[
\ind^L_P(\abs{f})  \leq \ind_P^K(\abs{f}).
\]
\end{corollary}
\begin{proof}
Let $\pi \in K$ be a uniformizer. For each $n \in \NN$, consider the existentially $K$-definable set $B_n(P)$, given by
\[
B_n(P) = \{ x \in \VF^m \mid  \bigwedge_{i = 1}^m \ord(x_i-P_i) > n \ord(\pi) ) \}.
\]
We apply (the $\Loc^0_{K_0}$-version with $K=K_0$ of) Theorem \ref{th:K-ind-global} to $f_{|X \cap B_n(P)}$ and find that for all $n \in \NN$
\[ \ind^L_{X \cap B_n(P)}(\abs{f}) \leq \ind^K_{X \cap B_n(P)}(\abs{f}). \]
Now take the supremum over $n$ of both sides.
\end{proof}
The condition that $f$ is existentially definable can not be omitted in Theorem \ref{th:K-ind-global} or Corollary \ref{cor:K-ind-local}, as the following example illustrates.
\begin{example} \label{ex:descent-failure}
	Let $f \colon \VF \to \VF$ be the definable function given by
	\[
	x \mapsto f(x) =
	\begin{cases}
		1/x		& \text{ if } x \neq 0 \text{ and } \forall y \in K (y^2 \neq -1), \\
		0 					&  \text{else}. \\
	\end{cases}
	\]
	Note that $f$ is not existentially definable.
	A direct calculation shows $\ind^{\QQ_3}_{0}(\abs{f}) = 1$, while $\ind^{ \QQ_3( \sqrt{-1} ) }_{0}(\abs{f}) = + \infty$.
\end{example}
%

%\begin{example}
%Let $X$ be existentially definable and $L \geq K \geq \QQ_p$ a tower of finite field extensions.
%It is not necessarily the case that
%%
%\[ \int_{X_K} \, dx \leq \int_{X_L} \, dx .\]
%%
%Consider for example the set $X = \{ x \in \VF \mid \rv(x)  = 1 \}$.
%Then $X_K = 1 + \cM_{K}$, hence $X_{\QQ_p}$ has measure $p^{-1}$, while for an uramified extension $L$ of degree $d$, the set $X_L$ has volume $p^{-d}$.
%\end{example}
%
\begin{remark}
For any definable set $X$, we have a ring of ``constructible functions'' $\cC(X)$ on $X$, as defined in \cite{CH18}. If $X$ is existentially definable, it makes sense to consider the sub-semiring where we only allow existential $\LgDP$-formulas in the description of the generators and require generators of type $a,\alpha$ to only take values in $\bA_{+}$ (as in \cite[Se.~4.2]{CL08}) and $\VG_{\geq 0}$ respectively. In contrast to $\cC(X)$ (or $\cC_+(X)$, \cite{CL08,CGH21}), this semiring is not yet stable under integration. One would need to add several new types of functions. We give two examples below.
\begin{example}
Let $X \subseteq \VG \times \VG_{ > 0}$ be given by $\{ (\gamma,\delta) \mid 0 \leq 2 \gamma < \delta \}$ and consider the constant function $2 \in \cC(X)$. Integrating out the $\gamma$-variable, we obtain a new function $f(\delta) \in \cC(\VG_{> 0})$, given for each $K$ by
\[ f_K(\delta) 	= \sum_{0 \leq 2 \gamma < \delta} 2
=  e_K \delta + (e_K \delta \mod 2).  \]
This function is not of the form $e_K a_K(\delta)$ for any existentially definable $a \colon \VG_{> 0} \to \VG_{\geq 0}$. With some more work, one shows that $f(\delta)$ is not built up from the previously mentioned generators by addition and multiplication.
\end{example}
\begin{example}
Consider the existentially definable set $X \subseteq \VF \times \VG_{> 0}$, given by $\{(x,\delta) \mid 0 \leq \delta < 2 \ord x \}$ and the constant function $1 \in \cC(X)$. Then we compute that for all $\delta > 0$ and $K \in \Loc^0$ that
\[ g_K(\delta) \coloneqq \int_{X_K(\delta)} 1 \, dx = \begin{cases}
	q_K^{ - (\frac{e_K \delta}{2} + 1) }	& \text{ if } \delta \in  2 \Gamma, \\
	q_K^{- (\frac{e_K \delta+1}{2} )}    	& \text{ if } \delta \notin 2 \Gamma.
\end{cases}  \]
Now consider $f(\delta) = q^{e_K \delta +2} g(\delta)^2$. The latter is given, for $\delta \geq 0$ by
\[ f_K(\delta) = \begin{cases}
	1 	& \text{ if } \delta \in 2 \Gamma, \\
	q	& \text{ if } \delta \notin 2 \Gamma.
\end{cases} \]
This function $q^{ e_K (\delta \mod 2 \VG )  }$ is clearly not of the form $q^{e_K \beta(\delta)}$ for some existentially definable $\beta \colon \VG_{ > 0} \to \VG$ (else the complement of $2 \VG_{> 0}$ would be existentially definable). With some more work, one can show that it does not lie in the semiring generated by the previously mentioned generators.
\end{example}
\end{remark}
\subsection{Application to the (Serre)-Poincar\'e series}
We now consider the Poincar\'e series and Serre-Poincar\'e series associated to a tuple of polynomials $f(x) = (f_1(x),\dots,f_r(x) )$ in $m$ variables $x = (x_1,\dots,x_m)$ with coefficients in $\cO_{K}$. These generating series are associated to the number of (reductions of) solutions of $f(x) = 0$ in the residue rings. It is well known (see e.g. \cite[Thm.~2]{Meu81}, \cite{Den84,Den91}) that, up to a transformation of variables $T \mapsto q_K^{-s}$, these series are given by $p$-adic integrals.

We apply Theorem \ref{th:K-ind-global} to these integrals and find that the largest pole (in $s$) of these Poincar\'e series can only grow when passing to finite field extensions (Theorem \ref{th:serre-poincare}).
%This implies that, asymptotically, the proportion of solutions is of at least the same order of magnitude (Corollary \ref{cor:serre-poincare}).
%
%
\begin{definition} \label{def:poincare}
Let $K \in \Loc^0$ and take a tuple of polynomials $f(x) = (f_1(x),\dots,f_r(x) )$ in $m$ variables $x = (x_1,\dots,x_m)$ with coefficients in $\cO_{K}$.
Let $\pi \in K$ be a uniformizer and define for all $n \in \NN$ the numbers (as in \cite{Den84})
\begin{enumerate}
	\item $\tilde{N}_{n,K}(f) \coloneqq \# \{ \xi  \in  \left( \cO_K/(\pi^n) \right)^m \mid  f(\xi) = 0 \}$,
	\item $N_{n,K}(f) \coloneqq \# \{ \xi \in \left( \cO_K/(\pi^n) \right)^m \mid \exists y \in \cO_K^m(f(y) = 0 \land \xi =  y + (\pi^n)^m  )  \}$.
\end{enumerate}
We call the corresponding generating series
\begin{enumerate}
	\item $\tilde{P}_{f,K}(T) \coloneqq \sum_{n = 0}^{+ \infty} \frac{ \tilde{N}_{n,K}(f) }{q^{nm}} T^n$,
	\item $P_{f,K}(T)  		  \coloneqq \sum_{n = 0}^{+ \infty} \frac{N_{n,K}(f) }{q^{n(m+1)}} T^n$,
\end{enumerate}
the \emph{Poincar\'e series} and \emph{Serre-Poincar\'e series} of $f$, respectively.
\end{definition}
\begin{theorem} \label{th:serre-poincare}
Let $K_0$ be a $p$-adic field and $f(x) = (f_1(x),\dots,f_r(x))$ be tuple of polynomials in $m$ variables $x = (x_1,\dots,x_m)$ with coefficients in $\cO_{K_0}$. For all $K \in \Loc^0_{K_0}$ we write $- \lambda_K(f)$ (resp. $- \tilde{\lambda}_K(f)$) for the largest real part over all poles in $s$ of $P_{f,K}(q_K^{-s})$ (resp. $\tilde{P}_{f,K}(q_K^{-s})$). Let $\lambda_{K}(f) = + \infty$ (resp. $\tilde{\lambda}_{K}(f) = + \infty$) if there are no poles. For any finite field extension $L \geq K$, we have
\[
\lambda_{L}(f) \leq \lambda_K(f) \quad \text{ and } \quad \tilde{\lambda}_L(f) \leq \tilde{\lambda}_K(f).
\]
\end{theorem}
\begin{proof}
Let $X$ be the existentially definable set
\[
\begin{array}{ll}
	\{ (x,w) \in \VF^m \times \VF \mid 	&\exists y \in \VF^m ( \ord(y) \geq 0 \land f(y) = 0 \\
						     			&\land \min_{i}(\ord(x_i-y_i)) \geq \ord w \geq 0  ) \}.
\end{array} \]
%Here, $\ord(y)$ denotes the minimal valuation over all entries of the tuple $y$.
By \cite[Lemma.\,3.1]{Den84}, we have that
\begin{equation} \label{eq:integral-form}
	P_{f,K}(q_K^{-s}) = \frac{q_K}{q_K - 1} \int_{X_K} \abs{w}_K^s \, dx dw,
\end{equation}

for all real $s$ for which $P_K(q_K^{-s})$ is finite. Now recall that $P_{f,K}(T)$ is rational (\cite[Thm.\,1.1]{Den84} and \cite[Cor.~4.5.2]{CH18}). Hence, the real part of its pole closest to the origin ($=q^{\lambda_K(f)}$) equals its radius of convergence. Now $-\lambda_{K}(f)$ can also be computed as the infimum over all $s \in \RR$ for which the right-hand side of (\ref{eq:integral-form}) is finite. We thus observe that $\lambda_K(f) =  \ind^K_X(\abs{w^{-1}})$.
Theorem \ref{th:K-ind-global} now implies that $\lambda_{L}(f) \leq \lambda_K(f)$.

The claim for $\tilde{\lambda}_K(f)$ similarly follows from the the following identity (\cite[Thm.~2]{Meu81}, \cite{Den91})
\[
(1- q_K^{-s}) \tilde{P}_{f,K}(q_K^{-s}) = 1 - q_K^{-s} \int_{\cO_K} \abs{f(x)}_K^s  \, dx ,
\]
where $\abs{f(x)}_K = \max_{j = 1}^r\abs{f_j(x)}_K$.
\end{proof}
\begin{proof}[Proof of Theorem \ref{th:intro-Serre-Poincare}]
This is part of Theorem \ref{th:serre-poincare}.
\end{proof}
As mentioned in the introduction, Theorem \ref{th:serre-poincare} implies asymptotic comparisons for the the numbers $N_{n,K}(f)$.
\begin{corollary} \label{cor:serre-poincare}
Let $K_0 \in \Loc^0$ and $f(x) = (f_1(x),\dots,f_r(x))$ be a tuple of polynomials in $m$ variables $x = (x_1,\dots,x_m)$ with coefficients in $\cO_{K_0}$.
For any tower of finite field extensions $K_0 \leq K \leq L$, we have that
\[
\limsup_{n \to \infty}\left( \frac{ N_{n,K}(f) }{ q_K^{n(m+1)} } \right)^{\frac{1}{n}}
\leq \limsup_{n \to \infty}\left( \frac{ N_{n,L}(f) }{ q_L^{n(m+1)} } \right)^{\frac{1}{n}}.
\]
\end{corollary}
\begin{proof}
Since we established $q_K^{\lambda_K(f)}$ as the radius of convergence of $P_{f,K}(T)$, this follows from the above Theorem \ref{th:serre-poincare}.
\end{proof}
%
%\begin{remark}
%Theorem \ref{th:serre-poincare} also holds for the Serre-Poincar\'e series counting simultaneous solutions of polynomials $f_1(x),\dots,f_k(x)$.
%\end{remark}

%Finally, one may expect that our semi-continuity results for $K$-indices and largest poles imply corresponding results over non-Archimedean local fields of large positive characteristic, by variants of the transfer principles for $p$-adic integrals from e.g.~\cite{CL10,CGH14,CGH18}.
%Indeed, for large positive characteristic, the integrability questions are reduced to summability issues over subsets of $\RV^n$ for some $n\ge0$, which depend only on the residue field and not on the local field, by typical transfer principles for $p$-adic integrals.

\subsection{Non-Archimedean local fields of large positive characteristic and a transfer principle}\label{sec:7.4}
The existing transfer results for $p$-adic integrals (see e.g.~\cite{CL10,CGH14,CGH18}) imply a corresponding transfer statement in the current setting. This in turn implies descent for the $K$-index in non-Archimedean fields of sufficiently large positive characteristic.

First, we extend some notation. For any $\LRV$-definable set $X$, one may take a defining formula $\varphi(x)$ and consider the sets $\varphi(\FF_q((t)))$, for any prime power $q = p^f$.
By the compactness theorem, any other choice of formula $\psi(x)$ will yield the same set, as soon as $p$ is sufficiently large (where ``sufficiently large'' may depend on $\varphi$ and $\psi$).
For sufficiently large $p$, extend the notation from Section \ref{sec:modtheo-conventions} for definable sets by writing $ \varphi(\FF_q((t))) \eqqcolon X_{\FF_{q}((t))}$.
Similarly extend the notation for definable functions.

For the statement of the next lemma, note that if two non-Archimedean local fields $K,L$ have isomorphic residue fields, then this induces an isomorphism $\RV_K \cong \RV_L$, respecting the relations $P_{1,d}$, $\mid$ and $\oplus$ from Section \ref{sec:modtheo-conventions}.
\begin{lemma} \label{le:transfer}
Let $X \subseteq \VF^m \times \RV^n$ be $\LRV$-definable. Then for all non-Archimedean local fields $K$ of sufficiently large residue characteristic and all $\xi \in \RV_K^n$ the value of
\[ \int_{X_K(\xi)} \, dx  \]
depends only $\xi$ and the isomorphism class of the residue field of $K$.
%Let $X \subseteq \VF^m \times \RV^n$ be an $\LRV$-definable set. Then there exists some prime number $p$ such that the following holds. Let $K,L$ be two non-Archimedean local fields (of any characteristic) with isomorphic residue fields of characteristic at least $p$. Then there exists an isomorphism $\RV_K \cong \RV_L$, such that under this isomorphism it holds for all $\xi \in \RV^n_K = \RV^n_L$ that
%%
%\[ \int_{X_K(\xi)} \, dx = \int_{X_L(\xi)} \, dx.  \]
\end{lemma}
\begin{proof}
%First reduce to the case where $\varphi(x,\xi)$ contains no symbols involving $\RV_N$ or $\rv_N$ for $N > 1$ (e.g. take $p$ larger than the largest occuring such $N$ in $\varphi(x,\xi)$).
Choose some defining formula $\varphi(x,\xi)$ for $X$.
The (proof of) \cite[Prop 1.8]{R17}, produces an $\LgDP$-formula (Definition \ref{def:gdp}) $\psi(x,\zeta,\gamma)$ such that for all henselian valued fields of characteristic zero with angular component maps $\varphi(x,\xi)$ and $\psi(x,\zeta,\gamma)$ define the same set in $K^m \times R_1^n \times \Gamma^n$, up to the isomorphism $\RV^{\times} \cong R_1^{\times} \times \Gamma$ induced by a choice of angular component map on $K$. Since it suffices to consider fields whose residue characteristic is sufficiently large, we may assume that $\psi(x,\zeta,\gamma)$ does not contain any symbols $\ac_N$ or quantifiers over $\RF_N$ for $N > 1$ (thus it is an $\Ldp$-formula, in the notation of \cite{CGH14}).
%Moreover, we can take $\psi$  such that it does not contain any symbols $\ac_N$ or variables (either bound or free) over $\RF_N$ for $N > 1$.
By \cite[Theorem 4.4.3]{CGH14}, there is a fixed constructible function $g(\zeta,\gamma) \in \cC( \RF_1^r \times \VG^r )$ whose interpretation in all non-Archimedean local fields of sufficiently large positive characteristic is precisely $\int_{\psi(x,\zeta,\gamma)} \, dx$, whenever this integral is finite (by \cite[Thm.~4.4.1]{CGH14} finiteness of this integral depends only on $\zeta$, $\gamma$ and the isomorphism class of the residue field).

For each formula $\chi(\zeta,\gamma)$ occurring in the description of $g(\zeta,\gamma)$, there exists a finite disjunction over formulas of the form $ \chi_{\RF}(\zeta) \land \chi_{\VG}(\gamma)$, with $\chi_{\RF}(\zeta)$ an $\Lring$-formula and $\chi_{\VG}(\gamma)$ an $\Log$-formula, such that $\chi(\zeta,\gamma)$ is equivalent to this disjunction, for all Henselian valued fields of equicharacteristic zero with angular component maps (see e.g. \cite[Thm.~2.1.1]{CL08}).
By compactness, these equivalences also hold for Henselian valued fields of large residue characteristic with angular component maps.
Thus for all non-Archimedean local fields $K$ of sufficiently large residue characteristic, the value of $g_K(\zeta,\gamma)$ depends only $\zeta$, $\gamma$ and the isomorphism class of the residue field (since the value group is always isomorphic to $\ZZ$).
\end{proof}
\begin{proposition} \label{prop:ind-transfer}
Let $X \subseteq \VF^m$ and $f \colon X \to \VF$ be $\LRV$-definable. Then for all non-Archimedean local fields $K,L$ with isomorphic residue fields of sufficiently large characteristic it holds that
\[ \ind^K_X(\abs{f}) = \ind^L_X(\abs{f}) .\]
\end{proposition}
\begin{proof}
If $K$ and $L$ have isomorphic residue fields, then we already remarked that one may identify $\RV_K$ and $\RV_L$. Now Lemma \ref{le:transfer} implies that the volumes $a_K(\xi_0)$ and $a_L(\xi_0)$ in the proof of Theorem \ref{th:K-ind-global} are equal, for all $\xi_0 \in \RV_K = \RV_L$.
\end{proof}
Theorem \ref{th:K-ind-global} and proposition \ref{prop:ind-transfer} immediately imply descent in large positive characteristic. We conclude with the analogue of Theorem \ref{th:K-ind-global} for large positive characteristic and we leave the corresponding analogues of Corollary \ref{cor:K-ind-local} and Theorem \ref{th:serre-poincare} to the reader.
\begin{corollary} \label{cor:descent-pos-char}
Let $X \subseteq \VF^m$ and $f \colon X \to \VF$ be existentially $\LRV$-definable and let $K$ be a non-Archimedean local field of sufficiently large positive characteristic. For all finite field extensions $L \geq K$ it holds that
\begin{equation*}
\ind^L_X(\abs{f}) \leq \ind^K_X(\abs{f}).
\end{equation*}
\end{corollary}

Our descent results from Theorems \ref{th:K-ind-global} and \ref{th:serre-poincare} and Corollary \ref{cor:descent-pos-char} can be seen as semi-continuity results for the $K$-indices on the one hand, and, for the largest poles on the other hand. They open the way to study the invariant that comes up by taking their limits over larger and larger field extensions. It would be interesting to study these limits, and, to link them to complex invariants, whenever possible.

\bibliographystyle{amsalpha}
\bibliography{sources}

\end{document}